\documentclass{amsart}
\usepackage[english]{babel}
\usepackage{fourier} % Use the Adobe Utopia font for the document - comment this line to return to the LaTeX default
\usepackage{graphicx}
\usepackage{amscd}
\usepackage{amsmath}
\usepackage{amsfonts}
\usepackage{amssymb}
\usepackage{bbm}
\usepackage{setspace}
\usepackage{enumerate}         % better lists
\usepackage{fixme}
\usepackage{color}
\usepackage{url}
\usepackage{amsthm}
\usepackage{hyperref}
\usepackage{bm}
\usepackage{xy}

\theoremstyle{plain}
\newtheorem{theorem}{Theorem}[section]

\newtheorem{lemma}[theorem]{Lemma}

\newtheorem{prop}[theorem]{Proposition}

\newtheorem{definition}[theorem]{Definition}

\newtheorem{assumption}[theorem]{Assumption}
\theoremstyle{remark}
\newtheorem{remark}[theorem]{Remark}

\numberwithin{equation}{section}

\newcommand{\ind}{1\!\kern-1pt \mathrm{I}}
\newcommand{\rsto}{]\!\kern-1.8pt ]}
\newcommand{\lsto}{[\!\kern-1.7pt [}

\numberwithin{equation}{section}

\newcommand{\id}{\operatorname{id}}
 
\begin{document}
\title{An elementary proof of the reconstruction theorem}
\author{Harprit Singh and Josef Teichmann}
\address{ETH Z\"urich, D-Math, R\"amistrasse 101, CH-8092 Z\"urich, Switzerland}
\email{jteichma@math.ethz.ch}
\thanks{The authors gratefully acknowledge the support from ETH-foundation.}
\curraddr{}
\begin{abstract}
The reconstruction theorem, a cornerstone of Martin Hairer's theory of regularity structures, appears in this article as the unique extension of the explicitly given reconstruction operator on the set of smooth models due its inherent Lipschitz properties. This new proof is a direct consequence of constructions of mollification procedures on spaces of models and modelled distributions: more precisely, for an abstract model $Z$ of a given regularity structure, a mollified model is constructed, and additionally, any modelled distribution $f$ can be approximated by elements of a universal subspace of modelled distribution spaces. These considerations yield in particular a non-standard approximation results for rough path theory. All results are formulated in a generic $(p,q)$ Besov setting.
\end{abstract}

\maketitle
\section{Introduction}
Undoubtedly the reconstruction theorem is the single most important fundamental result of the theory of regularity structures: given a family of local expansions of a possibly existing distributional object, this theorem allows to actually reconstruct this object. So far proofs use deep results from wavelet analysis, see~\cite{Hairer2014} and ~\cite{Hairer2017}, or results from para-controlled distributions, see~\cite{Gubinelli2015} or~\cite{MP18}, or semigroup approaches, see~\cite{OSSW18}. By means of the reconstruction operator one can lift many questions on classes of distributions needed for solving certain singular equations, for instance ubiquitous existence (maybe up to renormalisation), uniqueness and regularity, to questions on local expansions. This has been successfully applied for singular stochastic partial differential equations like the KPZ equation or the $\Phi^4_3$ equation, where instead of solving an appropriately re-normalized version on a space of Schwartz distributions, one solves a lifted version on a space of modelled distributions and re-constructs from this abstract solution by the reconstruction operator a Schwartz distribution. The reconstruction operator is also an important tool to provide existence proofs, see~\cite{LPT18} and the references therein for recent applications, in particular also in the case $ \gamma < 0 $. Notice that we do not show existence of the reconstruction operator in this case here.

Given a degree $ \gamma > 0 $ the reconstruction operator has two main input slots: a model $Z$ and a modelled distribution $f$. It is a striking fact that the reconstruction operator depends in a Lipschitz manner on models and modelled distributions, however, notice that the space of models $ \mathcal{M}_{\mathcal{T}} $ is non-linear and the \emph{linear} space of modelled distributions $\mathcal{D}^\gamma$ depends crucially on the given model. Whence the reconstruction operator is defined on a "bundle", which we shall call in sequel the \emph{modelled distribution bundle}, as a Lipschitz map. If the underlying model is \emph{smooth}, i.e.~all involved distributions are actually smooth functions, then the reconstruction operator can be given explicitly (even for continuous models, e.g.~models where the distributions are actually continuous functions). The deep part of the results is the fact that the re-construction operator exists beyond smooth models.

We shall add the following observations to the fundamentals of the theory of regularity structures:
\begin{itemize}
\item Given a model $ Z $ we can construct mollified models $Z^\epsilon$ (which are smooth) in a canonical way. These converge under a slight modification of the involved topologies to $Z$.
\item We can globally describe a dense subset of modelled distributions for \emph{any} model, which can be continuously deformed along Lipschitz continuous curves of models.
\item It can be proved \emph{directly} that the explicitly given reconstruction operator is Lipschitz continuous on the  modelled distribution bundle restricted to smooth models. It is therefore possible by elementary means to extend the reconstruction operator from smooth models to general model. In contrast to the so far existing proofs of the reconstruction theorem we do not need deep results from wavelet analysis to establish this extension. 
\end{itemize}
For the readers who might find Section \ref{section 1} lengthy we point directly to the main result: Theorem \ref{thm_approximating_smooth_models}. We also point out that Section \ref{section 1} and \ref{subsec_dense_subset} can be read independently of each other and the rest of the article.
 
\section{Mollification in regularity structures}\label{section 1}
In this section we present some basic definitions of regularity structures, see~\cite{Hairer2014} and \cite{Hairer2017} for all further details, and we show how generic models can be mollified. This will be reminiscent of the lifting the action of singular kernels to modelled distributions, however, with some essential differences.
\subsection{Basic Setup for regularity structures}
Let $\mathcal{T}=(A,T,G)$ be a regularity structure and denote by $\mathcal{M}_{\mathcal{T}}$ the space of all models for $\mathcal{T}$. Recall that a model $Z\in \mathcal{M}_\mathcal{T}$ is a pair $Z=(\Pi, \Gamma)$, consisting of 
\begin{enumerate}
\item a map $\Pi: \mathbb{R}^d\to L(T, \mathcal{D}'(\mathbb{R}^d)),\ x\mapsto \Pi_x$, such that 
\[
\|\Pi\|_\gamma:=\sup_{x\in \mathbb{R}^d} \sup_{\zeta\in A\cap (-\infty,\gamma) } \|\Pi\|_{x,\zeta}<\infty
\]
for all $\gamma\in\mathbb{R}$, where $$\|\Pi\|_{x,\zeta}:= \sup_{\tau \in T_\zeta}\sup_{\varphi\in \mathfrak{B}^r} \sup_{\delta<1} \frac{|\langle \Pi_x \tau, \varphi^\delta_x \rangle |}{|\tau| \delta^\zeta} \, ,$$
for given $ x \in \mathbb{R}^d $ and $ \zeta \in A $,
\item a map $ \Gamma: \mathbb{R}^d\times \mathbb{R}^d \to G, \ (x,y)\mapsto \Gamma_{x,y}$, which satisfies the algebraic conditions
$$\Pi_x\Gamma_{x,y} = \Pi_y, \; \Gamma_{x,y} \Gamma_{y,z} = \Gamma_{x,z} $$ 
for all $ x,y,z \in \mathbb{R}^d $, and the analytic condition 
$$
\|\Gamma\|_{\gamma}:=\sup_{ \|x-y\|_s<1} \sup_{\zeta\in A\cap (-\infty,\gamma)}  \|\Gamma\|_{x,y,\zeta}<\infty
$$ 
for all $\gamma\in\mathbb{R}$, where
$$\|\Gamma\|_{x,y,\zeta}:=\sup_{A\ni \beta<\zeta}\sup_{\tau \in T_\zeta} \frac{|\Gamma_{x,y}\tau|_\beta}{|x-y|_{\mathfrak{s}}^{\zeta-\beta}|\tau|} \, ,$$ 
for $ x,y \in \mathbb{R}^d $ and $ \zeta \in A $.
\end{enumerate}

The space $\mathcal{M}_\mathcal{T}$ has a natural metric induced by the semi-norms\footnote{Note that $\mathcal{M}_\mathcal{T}$ is not a vector space, thus strictly speaking these are not semi-norms, but $\mathcal{M}_\mathcal{T}$ can be embedded in a vector space, where these are indeed semi-norms.} $\|Z\|_\gamma:=\|\Pi\|_\gamma+\|\Gamma\|_\gamma$ given by
$$ d(Z,\bar{Z})=\sum_{n\in \mathbb{N}} \frac{\|Z-\bar{Z}\|_n}{2^{n}(1+\|Z-\bar{Z}\|_n) },$$ for $Z,\bar{Z}\in \mathcal{M}_\mathcal{T}.$ We call a model $Z=(\Pi,\Gamma)\in \mathcal{M}_\mathcal{T}$ smooth, if for all $x\in \mathbb{R}$ and all $\tau\in T$ the actual distributions $\Pi_x\tau$ are smooth functions.

\begin{remark}
Note that this definition is from \cite{Hairer2017} and differs from the one in \cite{Hairer2014} by imposing global bounds instead of local ones. Of course all of our results can be adapted to the setting of \cite{Hairer2014} by simply keeping track of the compact sets on which the relevant bounds hold.
\end{remark}

\subsection{A natural question}
Given a regularity structure $\mathcal{T}$ are smooth models, i.e.~models where all involved distributions are actually smooth functions, dense in $\mathcal{M}_\mathcal{T}$?

The answer is no, as the following simple example on $\mathbb{R}$ illustrates:
\begin{itemize}
\item Let $\alpha\in (0,1)$ and $A=\{0,\alpha\}$,
\item let $T=T_0\bigoplus T_\alpha$ where$T_0=\langle \mathbf{1}\rangle$ and $T_{\alpha}=\langle\tau\rangle$,
\item and let $G$ consist of all linear maps satisfying $\Gamma \tau-\tau \in T_0$ and $\Gamma \mathbf{1}=\mathbf{1}$. (These are just all maps that satisfy the axioms for a structure group).
\end{itemize}
 
Now denote by $L^\infty$ the space of all bounded (measurable) maps on $\mathbb{R}$ to itself and by $C^\alpha_0$ the space of all continuous  functions $h$, such that $h(0)=0$ and $\sup_{|t-s|<1}\frac{|h(s)-h(t)|}{|t-s|^\alpha}<\infty$.  We have the following characterization of $\mathcal{M}_\mathcal{T}$ for the above regularity structure:
\begin{prop}
The map $$L^\infty\times C^\alpha_0\to \mathcal{M}_\mathcal{T},\  (f,h)\mapsto Z^{f,h}=(\Pi^{f,h},\Gamma^{h}),$$
where 
$$\Pi_x^{f,h} \mathbf{1} (y)=f(y), \ \Pi_{x}^{f,h}\tau(y)=(h(y)-h(x))f(y),$$
and 
$$\Gamma_{x,y}^{h}\tau= \tau-(h(x)-h(y))\mathbf{1}$$
is a bi-Lipschitz homeomorphism.   
\end{prop}

The well known fact that smooth functions are not dense in $\mathcal{C}^\alpha$, which also holds for $\mathcal{C}^\alpha_0$, yields a counter example to the density of smooth models in $\mathcal{M}_\mathcal{T}$. This construction also yields a counterexample if one imposes the natural condition $\Pi\mathbf{1}=1$, which corresponds to $f=1$. Thus, in general smooth models are not dense in $\mathcal{M}_\mathcal{T}$, at least with respect to the natural topology. In the next section we introduce a slightly weaker topology, where this deficiency is cured.

\subsection{An slightly weakened topology for models}\label{subsec topologies}
Recall the fact that any $\mathcal{C}^{\alpha}$ function $f$ (or distribution, if we allow for $\alpha<0$) can be approximated by smooth functions in the $\mathcal{C}^{\alpha-\epsilon}$-norm and furthermore the approximating sequence can be chosen such that its $\mathcal{C}^\alpha$-norm is controlled by that of $f$. We shall introduce a similar result for models of a regularity structure. We shall work with the following standing assumption, which gives a special role to polynomials:

\begin{assumption}\label{ass_polynomial_reg_structure}
Let $\mathcal{T}=(A,T,G)$ be a regularity structure and fix a dimension $d$. We assume that for any integer $n\in\mathbb{N}$ $ \bar{T}_n = T_n $, where $\bar{T}=\bigoplus_{n\in \mathbb{N}} \bar{T}_n$ denotes the polynomial regularity structure of dimension $d$ and furthermore that $G$ acts on $\bar{T}$ in the canonical way.
\end{assumption}

\begin{remark} 
Let us point out that one can always add the polynomials to a regularity structure and thus this assumption is not restrictive. Since in applications we often work with non-polynomial regularity orders $ \alpha - \epsilon$, for arbitrary small $ \epsilon > 0 $, also the restriction that $ T_n $ only consists of polynomials does not matter. For our purpose, however, we could also just assume that $ \bar{T}_n \subset T_n $, but in order to stay in line with the literature and make notation lighter we chose the above slightly more restrictive assumption.
\end{remark}

\begin{definition}\label{def_appropriate_seminorms}
For $\epsilon>0$ we introduce on models $Z=(\Pi,\Gamma)\in \mathcal{M}_\mathcal{T}$ the following semi norms:
\begin{enumerate}
\item  Let $$\|\Pi\|_{\gamma,\epsilon}=\max\Bigg\{ \sup_{x\in \mathbb{R}^d} \sup_{\zeta\in (A\setminus \mathbb{N})\cap (-\infty,\gamma) } \|\Pi\|_{x,\zeta-\epsilon},\; \sup_{x\in \mathbb{R}^d} \sup_{\zeta\in \mathbb{N}\cap (-\infty,\gamma) } \|\Pi\|_{x,\zeta}  \Bigg\},$$ 
\item and similarly 
$$\|\Gamma\|_{\gamma,\epsilon}:=\max\Big\{ \sup_{ \|x-y\|_s<1} \sup_{\zeta\in (A\setminus \mathbb{N})\cap (-\infty,\gamma)}  \|\Gamma\|_{x,y,\zeta-\epsilon},\; \sup_{ \|x-y\|_s<1} \sup_{\zeta\in \mathbb{N}\cap (-\infty,\gamma)}  \|\Gamma\|_{x,y,\zeta}\Big\}. $$
\end{enumerate}
and write $\|Z\|_{\gamma,\epsilon}=\|\Pi\|_{\gamma,\epsilon}+\|\Gamma\|_{\gamma,\epsilon}$.
\end{definition}
\begin{remark}
Note that these semi-norms are defined in such a way that they agree on the polynomial part of the model, while on the rest they are slightly weaker.
\end{remark}
\begin{remark}
Let us us mention that in applications of regularity structures the models one usually works with are actually contained in the closure of smooth models in $\mathcal{M}_T$ by direct inspection. We shall prove that this is a generic fact (with respect to the above weaker topologies).
\end{remark}

\subsection{Mollifying a model}\label{subsec mollifying}

Let $\phi$ be a smooth compactly supported function whose integral equals $1$. We introduce the linear map $$J: \mathbb{R^d}\to L(T,T),$$ such that $J(x)$ annihilates polynomials for each $x\in \mathbb{R}$ and, for $\alpha\in (A\setminus \mathbb{N})$, maps $\tau\in T_\alpha$ to 
$$J(x)\tau=\sum_{|k|_{\mathfrak{s}}<\alpha} \frac{X^k}{k!} D^k (\phi * \Pi_x\tau) (x) \, .$$
Notice that the map $J$ depends on the choice of $\phi$. We say that the tuple of maps $\tilde Z=(\tilde \Pi,\tilde \Gamma)$ is obtained from $Z=(\Pi,\Gamma)$ by \emph{mollification with $\phi$} if it is of the following form
\begin{enumerate}
\item For $\alpha\in A\setminus \mathbb{N}$ and $\tau\in T_\alpha$ one has
$$\tilde{\Pi}_x \tau:=\phi* \Pi_x \tau-\Pi_x J(x) \tau \, ,$$ 
and on polynomials $\tilde{\Pi}$ agrees with $\Pi$.
\item The following identity holds
$$\tilde{\Gamma}_{x,y}:=\Gamma_{x,y}+J(x)\Gamma_{x,y}-\Gamma_{x,y} J(y) \, .$$
\end{enumerate}

The next lemma shows that mollified models are models and satisfy natural bounds.

\begin{lemma}\label{lem_mollification_bounds}
Let $Z=(\Pi,\Gamma)$ be a model satisfying Assumption \ref{ass_polynomial_reg_structure}, then the mollification $\tilde Z=(\tilde \Pi,\tilde \Gamma)$ is a model. Furthermore for $\gamma\in \mathbb{R} $ the bound 
$$\|Z\|_\gamma\lesssim \|\tilde{Z}\|_\gamma$$
holds uniformly over $\phi\in \big\{ \psi^\lambda\in \mathcal{C}_c^\infty \; | \, \psi \in \mathfrak{B}_r, \; \lambda\in (0,1)\big\}$ for $r>|\min A|\vee \gamma$. 
\end{lemma}
\begin{remark}
Let us point out that this definition of mollification is very canonical. Indeed it is the minimal "correction" one has to make to the the expression $\phi*\Pi$ in order to obtain a model. This is why similarities to the study of Greens-kernels in regularity structures are present, see Remark \ref{Remark Hairer} below.
\end{remark}

For the proof of this lemma the following version of Taylor's formula will be useful, see \cite[Proposition A1.]{Hairer2014}. We equip $\mathbb{N}^d$ with the partial order where $k\geq l$ if $k_i\geq l_i$ for all $i$ and write $k_<:=\{l\in \mathbb{N}^d | l\leq k , l\neq k\}$. Furthermore set $\mathfrak{m}(k) := \inf\{ i| l_i\neq 0 \}$.
\begin{lemma}\label{lem_taylor_formula}
Let $A\subset \mathbb{N}^d$ be such that $k\in A\Rightarrow k_<\subset A$ and define $\partial A=\{k\notin A | k-e_{\mathfrak{m}(k)} \in A\}.$ Then the identity $$f(x)=\sum_{k\in A} \frac{D^kf(0)}{k!}h^k+\sum_{k\in\partial A} \int_{\mathbb{R}^d} D^kf(h) \mu_k(x,dh) \, ,$$
where $\mu^k(x,dh)$ are signed measures on $\mathbb{R}^d$ supported on $\{ z\in \mathbb{R}^d| z_i\in [0,x_i]\}$ with total mass $\frac{x^k}{k!}$, holds for every smooth function $f$ on $\mathbb{R}^d$.
\end{lemma} 
\begin{proof}[Proof of Lemma \ref{lem_mollification_bounds}]
First we check that the algebraic constraint $\tilde \Pi_x \tilde \Gamma_{x,y}=\tilde\Pi_y$ is satisfied: indeed for polynomials this follows directly form the definition. The only issue arises, when $\tau\in T_\alpha$ for $\alpha>0$ not an integer. Then we see that
\begin{align*}
\tilde \Pi_x \tilde \Gamma_{x,y}\tau &= \tilde \Pi_x \big(\Gamma_{x,y}\tau+J(x)\Gamma_{x,y}\tau-\Gamma_{x,y} J(y)\tau\big) \\
& = \underbrace{\tilde \Pi_x \Gamma_{x,y}\tau}_{\phi* \Pi_y\tau-\Pi_x J(x) \Gamma_{x,y}\tau }
+ \Pi_x  J(x)\Gamma_{x,y}\tau-\Pi_y J(y)\tau \\
&=\phi* \Pi_y\tau-\Pi_y J(y)\tau \\
&=\tilde\Pi_y \tau \, .
\end{align*}
In an analogous manner the second algebraic condition is checked: 
\begin{align*}
\tilde \Gamma_{x,y} \tilde \Gamma_{y,z} \tau & = \Gamma_{x,y} (\Gamma_{y,z} \tau + J(y)\Gamma_{y,z} \tau -\Gamma_{y,z} J(z)\tau) + \\
& + J(x) \Gamma_{x,y} (\Gamma_{y,z} \tau + J(y)\Gamma_{y,z} \tau -\Gamma_{y,z} J(z)\tau) - \\
& - \Gamma_{x,y} J(y) (\Gamma_{y,z} \tau + J(y)\Gamma_{y,z} \tau -\Gamma_{y,z} J(z)\tau) \\
& = \Gamma_{x,z} + J(x) \Gamma_{x,z} \tau - \Gamma_{x,z} J(z) \tau \, ,
\end{align*}
by applying the algebraic properties of the $ \Gamma $ maps, and using that $ J $ vanishes on polynomials.

Before we start showing the analytic bounds let us make the following simple observation: let $\phi=\psi^\lambda$, then we have for $\tau \in T_\zeta$
\begin{align}\phi*\big(\Pi_x\tau\big)(y)&=\big(\Pi_x\tau\big)\psi^\lambda_y \nonumber \\
&=\big(\Pi_y\Gamma_{y,x}\tau\big)\psi^\lambda_y\nonumber\\
&=\sum_{\alpha<\zeta}\Pi_y Q_\alpha(\Gamma_{y,x})(\psi^\lambda_y)\nonumber\\
&\lesssim \sum_{\alpha<\zeta} \lambda^{\alpha}|x-y|^{\zeta-\alpha}\label{pointwise estimate} \, .
\end{align}
Now we can establish the analytic bounds, starting with the bounds on $\tilde{\Pi}_x \tau$ for $\tau\in T_\zeta$. Note that there are three different cases:
\begin{enumerate}
\item The case $\zeta\in \mathbb{N}$, i.e $\tau$ is a monomial. Then $ \tilde \Pi \tau = \Pi \tau $ and there is nothing left to show
\item $\zeta< 0$ and we consider $\tilde{\Pi}_x \tau(\varphi^\delta)$. If $\delta\leq \lambda$, we apply Equation~\eqref{pointwise estimate} and obtain 
$$|\tilde{\Pi}_x \tau(\varphi_x^\delta)|\leq\int \big|\phi*\Pi_x\tau(y)| |\varphi^\delta_x(y)|\; dy \lesssim \sum_{\alpha<\zeta} \lambda^{\alpha}\delta^{\zeta-\alpha} \lesssim \delta^{\zeta} \, . $$
If $\delta>\lambda$ we note that $\psi^\lambda* \varphi_x^\delta=\psi^\lambda_x* \varphi^\delta$
and thus we obtain as above 
\begin{equation}\label{the more difficult case}
|\tilde{\Pi}_x \tau(\varphi_x^\delta)|\lesssim \sum_{\alpha<\zeta} \delta^{\alpha}\lambda^{\zeta-\alpha} \lesssim \delta^{\zeta} \, .
\end{equation}
\item The case $\zeta\in \mathbb{R}_+\setminus\mathbb{N}$: By Lemma~\ref{lem_taylor_formula}, with $A=\{k\in\mathbb{N}^d \ | \ |k|_\mathfrak{s}<\zeta\}$, we have 
\begin{align*}
\tilde{\Pi}_x \tau(y)&=\phi*\big(\Pi_x\tau\big)(y)-\sum_{k\in A} \frac{(y-x)^k}{k!} D^k\big( \phi*\Pi_x\tau\big)(x)\\
&=\sum_{k\in \partial A} \int D^k\big( \phi * \Pi_x\tau \big)(h) d\mu_k(y-x,h)\\
&\lesssim \sum_{k\in \partial A} \sum_{\alpha<\zeta} \lambda^{\alpha-|k|_\mathfrak{s}} |x-y|_\mathfrak{s}^{\zeta-\alpha}\int d\mu_k(y-x,h)\\
&\lesssim \sum_{k\in \partial A} \sum_{\alpha<\zeta} \lambda^{\alpha-|k|_\mathfrak{s}} |x-y|_\mathfrak{s}^{\zeta-\alpha+|k|_\mathfrak{s}} \, .
\end{align*}

Thus for $\lambda>\delta$:
$$\big|\tilde{\Pi}_x \tau(\varphi^\delta)\big|\lesssim \sum_{k\in \partial A} \sum_{\alpha<|\tau|} \lambda^{\alpha-|k|_\mathfrak{s}} \delta^{\zeta-\alpha+|k|_\mathfrak{s}}\lesssim \delta^{\zeta}\, , $$
where we used $|k|>\alpha$ in the last line.

In the case $\lambda<\delta$ we can argue as follows:
\begin{align*}
|\tilde{\Pi}_x \tau(\varphi^\delta_x)|&\leq|\phi*\big(\Pi_x\tau\big)(\varphi^\delta_x)|+\sum_{k\in A}\frac{1}{{k!}} \int | \varphi^\delta_x(x)(y-x)^k| \,dy |D^k\big( \phi*\Pi_x\tau\big)(x)|\\
&=|\varphi^\delta*\big(\Pi_x\tau\big)(\phi^\lambda_x)|+\sum_{k\in A}\frac{1}{{k!}} \int |\varphi^\delta_x(x)(y-x)^k| \,dy |D^k\big( \Pi_x\tau\big)(\phi_x)|\\
&\lesssim \sum_{\alpha<\zeta} \delta^{\alpha}\lambda^{\zeta-\alpha} +\delta^{|k|}\lambda^{\zeta-|k|}
\lesssim \delta^{\zeta} \, ,
\end{align*}
where we have used Equation~\eqref{the more difficult case}.

\end{enumerate}
We check the analytic bound on $\tilde\Gamma_{x,y}\tau=\Gamma_{x,y}+J(x)\Gamma_{x,y}-\Gamma_{x,y} J(y)$, where again the only non trivial case is $\tau\in T_\zeta$ for $\zeta>0$ not being an integer. It suffices to check the relevant bound for $J(x)\Gamma_{x,y}-\Gamma_{x,y} J(y)$. For this we notice that 
\begin{align*}
J(x)\Gamma_{x,y}\tau&=\sum_{|k|<\alpha} \frac{X^k}{k!} D^k (\phi * \Pi_x\Gamma_{x,y}\tau) (x)\\
&=\sum_{|k|<\alpha} \frac{X^k}{k!} D^k (\phi * \Pi_y\tau) (x),
\end{align*}
and
\begin{align*}
\Gamma_{x,y}J(y)\tau&=\Gamma_{x,y} \sum_{|k|<\zeta} \frac{X^k}{k!} D^k (\phi * \Pi_y\tau) (y)\\
&= \sum_{|k|<\zeta} \frac{\Gamma_{x,y}X^k}{k!} D^k (\phi * \Pi_y\tau) (y)\\
&=\sum_{|k|<\zeta} \sum_{l+m=k} {k\choose l } \frac{(x-y)^l X^m}{k!}D^k (\phi * \Pi_y\tau) (y)\\
&=\sum_{|k|<\zeta} \sum_{l+m=k}  \frac{(x-y)^l}{l!}\frac{X^m}{m!} D^k (\phi * \Pi_y\tau) (y)\\
&=\sum_{l+m<\zeta}  \frac{(x-y)^l}{l!}\frac{X^m}{m!} D^{m+l} (\phi * \Pi_y\tau) (y) \, .\\
\end{align*}
\begin{enumerate}

\item First we look at the case $\lambda>|x-y|_\mathfrak{s}$, we have 
\begin{align*}
\big(J(x)\Gamma_{x,y}-\Gamma_{x,y} J(y)\big)\tau=&
\sum_{m:\, |m|_\mathfrak{s}<\zeta} \frac{X^m}{m!}\Big( (D^m \phi * \Pi_y\tau) (x)-\sum_{|l|_\mathfrak{s}<\zeta-|m|_\mathfrak{s}} \frac{(x-y)^l}{l!}D^{l} (D^m\phi * \Pi_y\tau) (y)\Big) \, .
\end{align*}
Now we can bound each summand separately using Taylor's formula with $A=A_m=\{l\ | \|l|_\mathfrak{s}<\zeta-|m|_\mathfrak{s} \}$ and obtain 
\begin{align}
|J(x)\Gamma_{x,y}-\Gamma_{x,y} J(y)|_m & \lesssim \sum_{k\in \delta A_m} \int D^k\big( D^m\phi * \Pi_x\tau \big)(h) d\mu_k(h)\nonumber\\
&\lesssim \sum_{k\in \delta A_m} \frac{1}{\lambda^{|m|_\mathfrak{s}+|k|_\mathfrak{s}}} \sum_{\alpha<\zeta} \lambda^{\alpha}|x-y|_\mathfrak{s}^{\zeta-\alpha} \int d\mu_k(h)\nonumber\\
&\lesssim \sum_{k\in \delta A_m} \sum_{\alpha<\zeta} \lambda^{\alpha-(|m|_\mathfrak{s}+|k|_\mathfrak{s})}|x-y|_\mathfrak{s}^{\zeta-\alpha+|k|_\mathfrak{s}}\nonumber\\
&\lesssim |x-y|_\mathfrak{s}^{\alpha-(|m|_\mathfrak{s}+|k|_\mathfrak{s})+\zeta-\alpha+|k|_\mathfrak{s}}=|x-y|_\mathfrak{s}^{\zeta-|m|_\mathfrak{s}}\label{bound on Gamma}
\end{align}
\item The case $\lambda<|x-y|_\mathfrak{s}$ is easier, since we can bound the terms $J(x)\Gamma_{x,y}\tau$ and $\Gamma_{x,y}J(y)\tau$ separately:
First we write 
\begin{align*}
J(x)\Gamma_{x,y}=& \sum_{\alpha<\zeta} J(x)Q_{\alpha}\Gamma_{x,y} \\
=& \sum_{\alpha<\zeta} \sum_{|k|_\mathfrak{s}<\alpha} \frac{X^k}{k!} \big(\Pi_xQ_\alpha \Gamma_{x,y} \tau\big)\big( D^k\psi^\lambda_x\big)
\end{align*} 
and thus we obtain 
\begin{align}
|J(x)\Gamma_{x,y}\tau|_m &\lesssim \sum_{\alpha: \ |m|_\mathfrak{s}<\alpha<\zeta} \big|\big(\Pi_xQ_\alpha\Gamma_{x,y} \tau\big)\big(D^k\psi^\lambda_x\big)\big| \nonumber\\
&\lesssim \sum_{\alpha: \ |m|_\mathfrak{s}<\alpha<\zeta} |x-y|_\mathfrak{s}^{\zeta-\alpha} \lambda^{\alpha-|m|_\mathfrak{s}}\label{second gamma bound}\\
&\lesssim |x-y|_\mathfrak{s}^{\zeta-|m|_\mathfrak{s}}\nonumber \, .
\end{align}

Similarly we have
$$
\Gamma_{x,y}J(y)\tau=\sum_{|l|_\mathfrak{s}+|m|_\mathfrak{s}<\zeta}  \frac{X^m}{m!} \ \frac{(x-y)^l}{l!}  \big(\Pi_y\tau\big) D^{m+l} \psi^\lambda_x
$$
from which we conclude 
\begin{align}
|\Gamma_{x,y}J(y)\tau|_m&\lesssim \sum_{|l|_\mathfrak{s}<\zeta-|m|_\mathfrak{s}} |\frac{(x-y)^l}{l!}  \big(\Pi_y\tau\big)\big( D^{m+l} \psi^\lambda_x\big)|\nonumber \\
&\lesssim \sum_{|l|_\mathfrak{s}<\zeta-|m|_\mathfrak{s}} |x-y|^{|l|_\mathfrak{s}} \lambda^{\zeta-|m|_\mathfrak{s}-|l|_\mathfrak{s}}\label{third gamma bound}\\
&\lesssim|x-y|_\mathfrak{s}^{\zeta-|m|_\mathfrak{s}}\nonumber
\end{align}
\end{enumerate}
\end{proof}
\begin{remark} Note that the above proof yields for $|x-y|_\mathfrak{s}<\delta$ the pointwise estimate 
\begin{equation}\label{pointwise good estimate}
|\tilde{\Pi}_x \tau(y)|\lesssim |x-y|_\mathfrak{s}^\tau \, .
\end{equation}
\end{remark}

If we fix $\phi\in C_c^\infty$ such that $\int \phi=1$ and study the curve of models $Z^{\lambda}=(\Pi^{\lambda},\Gamma^{\lambda})$ which are obtained by mollifying $Z=(\Pi,\Gamma)$ by $\phi^\lambda$ we obtain the following result:
\begin{lemma}\label{lem_convergence_mollified_models}
For $\epsilon<\text{dist}\big((A\setminus \mathbb{N})\cap (-\infty,\gamma),\, \mathbb{N}\cap (-\infty,\gamma)\big)$ we have $\|Z-Z^\lambda\|_{\gamma,\epsilon}\lesssim \lambda^\epsilon\|Z\|_\gamma$, where the constant only depends on the regularity structure, $\gamma$ and $\phi$.
\end{lemma}

Since the proof is similar to the one of Lemma \ref{lem_mollification_bounds}, we have moved it to the appendix.

\begin{remark}\label{Remark Hairer} The definition of $J$ above is essentially identical to the one appearing in \cite{Hairer2014} when constructing a lift for singular kernels. Indeed, one could interpret convolving with elements of $\phi^\lambda$ as convolving with a $0$-regularizing kernel and showing appropriate continuity in the kernel. From this point of view there is a slight similarity to Theorem 4.2 in~\cite{HaiGer:2017}, where the authors construct simultaneously a lift for many different kernels and show appropriate continuity with respect to the kernels.
\end{remark}

By means of the previous lemmas we have finally proved the following approximation theorem:

\begin{theorem}\label{thm_approximating_smooth_models}
Let $(A,T,G)$ be a regularity structure satisfying Assumption \ref{ass_polynomial_reg_structure}. Let $Z\in \mathcal{M}_T$ be a model, which coincides on $\bar{T}\subset T$ with the canonical model. Then there exists a family of smooth models $\big(Z^\lambda\big)_{\lambda\in(0,1)}\in \mathcal{M}_T$ whose restriction to $\bar{T}$ coincides with the polynomial model, such that for any $\epsilon>0$ small enough,

$$\|Z-\mathbb{Z}^\lambda\|_{\gamma,\epsilon}\lesssim \lambda^\epsilon\|Z\|_\gamma$$ and
$$\|Z^\lambda\|_\gamma\lesssim \|Z\|_\gamma,$$
where the constants depends only on the regularity structure.
\end{theorem}

\subsection{Mollifying modelled distributions}\label{subsec_mollification_modelled_distribution}
Models form the groundwork for expansions in regularity structures. Expansions can be encoded by collections of expansion coefficients $(f(x))_{x \in \mathbb{R}^d}\in T$ with respect to a model. It is important to single out those expansion which actually correspond to a global object, which leads to the (Besov-type) definition of modelled distributions, cf. \cite[Definition 2.10]{Hairer2017}.
\begin{definition}\label{def modelled dist}
Given a regularity structure  $\mathcal{T}=(A,T,G)$ and a model for it $Z=(\Pi, \Gamma)$. For $p,q\in [1,\infty]$, we define $
\mathcal{D}^\gamma_{p,q}$ as the space of all measurable maps $f:\mathbb{R}^d\to T_{<\gamma}$, such that for all $\zeta \in A\cap (-\infty,\gamma)$ the following bounds hold:
\begin{itemize}
\item $\big\| |f(x)|_\zeta\big\|_{L^p} <\infty$
\item $\int_{h\in B^\mathfrak{s}_1} \Big\| \frac{|f(x+h)-\Gamma_{x+h,x}f(x)|_\zeta}{\|h\|^{\gamma-\zeta}_\mathfrak{s}}\Big\|^q_{L^p} \frac{dh}{\|h\|_\mathfrak{s}^{|s|_\mathfrak{s}}} =: \|f\|^q_{d^{\gamma,\zeta}_{p,q}}<+\infty \, .$
\end{itemize}
We define the norm $$\|f\|_{\mathcal{D}^\gamma_{p,q}}:=\sum_{\zeta \in A\cap (-\infty,\gamma)} \big\| |f(x)|_\zeta\big\|_{L^p}+\|f\|_{d^{\gamma,\zeta}_{p,q}}$$
on $\mathcal{D}^\gamma_{p,q}$.
It is possible to measure the distance between modelled distributions coming from different models. We introduce
\begin{align*}
\| f,\bar{f}\|_{\mathcal{D}^\gamma_{p,q}}&:=\sum_{\zeta \in A\cap (-\infty,\gamma)} \big\| |f(x)-\bar{f}(x)|_\zeta\big\|_{L^p}\\
&+\Bigg(\int_{h\in B^\mathfrak{s}_1} \Big\| \frac{|f(x+h)-\Gamma_{x+h,x}f(x)-( \bar{f}(x+h)-\bar{\Gamma}_{x+h,x}\bar{f}(x))|_\zeta}{\|h\|^{\gamma-\zeta}_\mathfrak{s}}\Big\|^q_{L^p} \frac{dh}{\|h\|_\mathfrak{s}^{|\mathfrak{s}|}}\Bigg)^{\frac{1}{q}}
\end{align*}
for two modelled distributions $f,\bar{f}$ and models $Z=(\Pi,\Gamma)$, $\bar{Z}=(\bar{\Pi},\bar{\Gamma})\in \mathcal{M}_\mathcal{T}$.
\end{definition}

Even though it would be desirable to find a direct mollification procedure on modelled distributions, this is unfortunately not possible \emph{without} applying the reconstruction operator. This is outlined in the next remark. However, due to the density statement of the next subsection, we do not need a mollification of modelled distributions.
\begin{remark}\label{Remakr mollifying modelled}
Suppose we are given a regularity structure $(A,T,G)$ satisfying Assumption \ref{ass_polynomial_reg_structure}, together with a model $Z=(\Pi,\Gamma)$. For a fixed $\phi$ we denote by $\tilde{Z}=(\tilde{\Pi},\tilde{\Gamma})$ the associated mollified model and let $J$ as above. We denote by $\tilde{\mathcal{D}}_{p,q}^\gamma$ the corresponding space of modelled distributions. Then there is a canonical mollification of modelled distributions, however, one \emph{needs} the reconstruction operator for it, see Theorem~\ref{thm_reconstruction} for its definition. One would hope that it was given by $\tilde{f}(\cdot):=f(\cdot)+J(\cdot)f(\cdot)$, since it has the natural property that it satisfies $$\tilde{\Pi}_x \tilde{f}(x)=\phi*\Pi_x f(x) \, .$$ Unfortunately, it turns out that this is in general not a modelled distribution for the model $\tilde{Z}$. The correct way of mollifying modelled distributions is instead given by the map $$\mathcal{D}^\gamma \to \tilde{\mathcal{D}}^\gamma,\ \, f(\cdot)\mapsto \tilde{f}(\cdot):=f(\cdot)+J(\cdot)f(\cdot)+(\mathcal{N}f)(\cdot),$$
where $(\mathcal{N}f)(x):=\sum_{|k|_\mathfrak{s}<\gamma} \frac{X^k}{k!} (\mathcal{R}f-\Pi_xf)(D^k\phi_x)$ and $\mathcal{R}$ denotes the so called reconstruction operator associated to the model $Z$. 

This definition is rather rigid, since any Ansatz of the form $$\tilde{f}(\cdot)=f(\cdot)+J(\cdot)f(\cdot)+P_{f,x}(\cdot),$$ where $P_{f,x}$ is some polynomial depending on f, leads to this choice. 

Let us record the following facts about the map $f\mapsto \tilde{f}$: 
\begin{enumerate}
\item It is continuous and indeed maps $\mathcal{D}_{p,q}^\gamma \to \tilde{\mathcal{D}}_{p,q}^\gamma$.
\item It is appropriate to call $\tilde{f}$ the $\phi$-mollification of $f$, since one has the identity $$\tilde{\mathcal{R}}\tilde{f}=\phi*\mathcal{R}f.$$
\item As we let the mollifier converge to the Dirac measure at $0$, we obtain convergence of the corresponding mollified modelled distributions in the $\mathcal{D}^{\gamma-\epsilon}_{p,q}$ topology.
\end{enumerate}
The proofs of these facts are straightforward modifications of the arguments in \cite[Section 5.2]{Hairer2014} for the case $p,q=\infty$ and \cite[Section 5]{Hairer2017} in the general case.

We conclude the remark by pointing out that the \emph{canonical} mollification of modelled distributions relies on the existence of the reconstruction operator. 
\end{remark}

\section{A globally defined dense subset of modelled distributions}\label{subsec_dense_subset}

Given a regularity structure  $\mathcal{T}=(A,T,G)$ and a model for it $Z=(\Pi, \Gamma)$ we can distinguish an important subspace of $\mathcal{D}^\gamma_{p,q} $ of elementary modelled distributions under the following natural assumption.
\begin{assumption} Throughout this section we assume that $\bar{T}\subset T$ is given by the polynomials and that we have a product $\star:\bar{T}\times T\to T$ and that $(\Pi, \Gamma)$ satisfy the following natural conditions hold for all $X^k\in \bar{T}$, $\tau\in T$:
\begin{enumerate}
\item For all $\Gamma\in G$: \ $\Gamma(X^k\star\tau)=\Gamma X^k\star \Gamma\tau$.
\item For all $x\in\mathbb{R}^d$: $\Pi_x (X^k\star\tau)= (\Pi_x X^k)\cdot (\Pi_x \tau)$, 
where $\cdot$ denotes the pointwise product.
\end{enumerate}
\end{assumption}
\begin{remark} Note that this is not a restriction, since any regularity structure $\mathcal{T}'$ and model $Z'$ can be extended to satisfy the assumptions above in such a way that the extension of $\mathcal{T}'$ is independent of the model $Z'$. Or put in algebraic terms: any regularity structure $ \mathcal{T}'$ and any model $ Z'$ allows for a tensor product with the polynomial regularity structure (of the correct dimension), which then satisfies canonically the above properties.
\end{remark} 

We define $ \mathcal{E}^\gamma_{p,q} \subset \mathcal{D}^\gamma_{p,q} $ of elementary modelled distributions. 
It contains all locally finite linear combinations of localized 'constants' \footnote{Notice the Remark~\ref{explicit_reconstruction} after Theorem~\ref{thm_reconstruction}.} e.g.
\[
x \mapsto f_{z,\phi,\tau} (x):=\Phi(x) \star\Gamma_{x,z} \tau
\]
for a smooth compactly supported function $ \phi$, where $\Phi$ denotes the canonical lift $\Phi(x):=\sum_k \frac{X^k}{k!} D^k\phi(x)$, and $y \in \mathbb{R}^d $ and $ \tau \in T $. Clearly $ x \mapsto f_{z,\phi,\tau} (x) $ lies in $ \mathcal{D}^\gamma_{p,q} $ since
\[
f_{z,\phi,\tau} (x) - \Gamma_{x,y} f_{z,\phi,\tau} (x)= \big( \Phi(x)-\Gamma_{x,y}\Phi(y)\big)\star\Gamma_{x,z}\tau \, .
\]

We have the following lemma:
\begin{lemma}\label{lem_density}
If $p,q<\infty$ let $\epsilon=0$ and $\epsilon>0$ otherwise, then the subspace $ \mathcal{E}^\gamma_{p,q} $ is dense in $ \mathcal{D}^\gamma_{p,q} $ in the $ \mathcal{D}^{\gamma-\epsilon}_{p,q}$-topology for $\epsilon$ small enough.
\end{lemma}
Since the proof of this lemma is rather straightforward but tedious, we deferred it to the appendix.

\section{The Reconstruction Theorem revisited}
One of the linchpins in the theory of regularity structures is the reconstruction theorem, see \cite{Hairer2014} and \cite{Hairer2017} for the respective versions:
\begin{theorem}\label{thm_reconstruction}
Let  $\mathcal{T}=(A,T,G)$ be a regularity structure and let $\alpha=\min A$. If $q=\infty$, let $\bar{\alpha}=\alpha$, else take $\bar{\alpha}<\alpha$. Then, for $\gamma>0$ and every model $Z=(\Pi,\Gamma)\in \mathcal{M}_\mathcal{T}$, there exists a unique continuous linear map  $\mathcal{R}_Z:\mathcal{D}^\gamma_{p,q}\to\mathcal{B}^{\bar{\alpha}}_{p,q}$, called reconstruction operator associated to $Z$, such that:
\begin{equation}\label{reconstruction cond}
\Bigg\| \Big\| \sup_{\eta\in \mathcal{B}^r} \frac{|\langle \mathcal{R}_Z f-\Pi_x f(x),\eta_x^\delta\rangle|}{\delta^\gamma} \Big\|_{L^p}\Bigg\|_{L^q_\delta} \lesssim \|f\|_{\mathcal{D}^\gamma_{p,q}} \|\Pi\|_\gamma (1+\|\Gamma\|_\gamma)
\end{equation}
uniformly over all $f\in \mathcal{D}^\gamma_{p,q}$ and all models $Z=(\Pi,\Gamma)$.
Furthermore, the map $Z\mapsto \mathcal{R}_Z$ is continuous in the following sense: If we denote by $\bar{Z}=(\bar{\Pi}, \bar{\Gamma})$ a second model for $\mathcal{T}$, $\mathcal{R}_{\bar{Z}}$ the associated reconstruction operator, and $\bar{\mathcal{D}}^{{\gamma}}_{p,q}$ the corresponding space, then the inequality
\begin{align}
\Bigg\| \Big\| \sup_{\eta\in \mathcal{B}^r} &\frac{|\langle \mathcal{R}_Z f-\Pi_x \bar{f}(x)-\mathcal{R}_{\bar{Z}}\bar{f}-\bar{\Pi}_x f(x),\eta_x^\delta\rangle|}{\delta^\gamma} \Big\|_{L^p}\Bigg\|_{L^q_\delta} \nonumber\\
\lesssim & \|f,\bar{f}\|_{\mathcal{D}^\gamma_{p,q}} \|\Pi\|_\gamma (1+\|\Gamma\|_\gamma) 
+\|\bar{f}\|_{\mathcal{D}^\gamma_{p,q}} \big(\|\Pi-\bar{\Pi}\|_\gamma (1+\|\Gamma\|_\gamma)+
\|\bar{\Pi}\|_\gamma \|\Gamma-\bar{\Gamma}\|_\gamma\big)
\label{continuity bound}
\end{align}
holds uniformly over models $Z,\bar{Z}\in\mathcal{M}_\mathcal{T}$ and ${f}\in \mathcal{D}^{{\gamma}}_{p,q},\  \bar{f}\in \bar{\mathcal{D}}^{\gamma}_{p,q}$.
\end{theorem}
\begin{remark}\label{explicit_reconstruction}
There are two cases when reconstruction can be written down explicitly:
\begin{enumerate}
\item If $Z$ is a continuous model, then the reconstruction operator maps to functions and it holds that $ \mathcal{R} f (x) = (\Pi_x f(x))(x) $ for $ x \in \mathbb{R}^d $. This was observed in \cite{Hairer2014}, where the relevant bounds follow from the the proof in the general (distribution) setting. See Proposition~\ref{prop_reconstruction_for_smooth_models} below for an elementary proof of this fact.
\item If $Z$ is any model, then for $ z \in \mathbb{R}^d $, $ \tau \in T $, $ f_{z,\tau}(x) : = \Gamma_{x,z} \tau $ is re-constructed by $ \Pi_z \tau $. These modelled distributions are referred to as constants in this article. Notice that this means in particular that for localized 'constants' $ x \mapsto f_{z,\phi,\tau} (x):=\Phi(x) \star\Gamma_{x,z} \tau $ we have
\[
\mathcal{R}(f_{z,\phi,\tau}) = \phi \, \Pi_z \tau
\]
for all $ z \in \mathbb{R}^d$, $\tau \in T $ and smooth, compactly supported $ \phi $.
\end{enumerate}
\end{remark}
The aim of this section is to establish an elementary proof of this theorem. 

\subsection{The reconstruction theorem for smooth models}\label{subsec reconstruct smooth}

In this subsection we shall proof the reconstruction theorem for smooth models, which we state for later reference: \begin{prop}\label{prop_reconstruction_for_smooth_models}
For smooth models the reconstruction operator can be defined via
$$
\mathcal{R}f=\big(\Pi_xf(x)\big)(x) \,
$$
for $ x \in \mathbb{R}^d $, in particular the Bounds \eqref{reconstruction cond} and \eqref{continuity bound} hold true.
\end{prop}
We shall first present a proof of a less general version of this proposition first. This serves to illustrates the arguments at the heart of the proof more clearly and without being burdened by as much notational/technical clutter. Also it shows a simpler argument for the case $p=q=\infty$, which is the one usually used in practice. Lastly, it also allows to record important identities which will allow to streamline the general proof considerably.

Let us recall the following notations, many of which from \cite{Hairer2017}:
\begin{enumerate}

\item We define the rescaled grid $\Lambda_n:=\big\{(2^{-n\mathfrak{s}_1}k_1,...,2^{-n\mathfrak{s}_d}k_d)\ :\ k_i\in \mathbb{Z}\big\}$. For a constant $C>0$ we let $\Lambda^C_{n} = B(0,C2^{-n}) \cap \Lambda_n$ be a localized version of $ \Lambda_n $.
\item\label{notation} Let $\mathbb{1}$ be a smooth function with compact support, such that ${(\mathbb{1}_k})_{k\in\Lambda_0}$, where we use the notation $\mathbb{1}_k(\cdot):=\mathbb{1}(\cdot-k)$, is a smooth partition of unity adapted to cubes $B^1_k$ centred at the points $k\in \mathbb{Z}^d$.
\item We write for $n\in\mathbb{N},k\in \Lambda_n$: $\mathbb{1}^n_k(x)=\mathbb{1}(\frac{x_1-k_1}{2^{n\mathfrak{s}_1}},...,\frac{x_d-{k_d}}{2^{n\mathfrak{s}_d}})$ for the rescaled partition of unity adapted to the rescaled cubes centered at the points in $\Lambda_n$. We shall denote such a cube by $B^n_k$. 
\end{enumerate}
\begin{remark}
Note that this particular choice of partition of unity will never play a role, indeed any family smooth partitions of unity $\{\phi_k^n\}_{k\in\Lambda_n}$ adapted to the cubes $B_k^n$ satisfying the bound 
$$ \|D^m \phi_k^n\|_{L^\infty} \lesssim 2^{n |m|_\mathfrak{s}}$$
uniformly over $k\in\Lambda, m\in\mathbb{N}^d, n\in\mathbb{N}$ would do the part. (Of course the differentiability condition could be relaxed too.)
\end{remark}
\begin{enumerate}
\item[(4)] Given a sequence $(a_k)_{k\in\Lambda_n}\in\mathbb{R}$ indexed by the rescaled grid, we write 
$$ |a|^p_{l^p_n}:=\sum_{k\in \Lambda_n} |a_k|^p \frac{1}{2^{|\mathfrak{s}|n}}.$$
\item[(5)] For any function $f:\mathbb{R}^d\to T$ define $$\bar{f}^n:\Lambda_n\to T_{\gamma},\   \bar{f}^n(x)=\frac{1}{|Q^n_x|}\int_{Q^n_x} \Gamma_{x,y}f(y)dy=\int_{Q^n_x} 2^{n|\mathfrak{s}|}\Gamma_{x,y}f(y)dy \, .$$
\end{enumerate}
We record the following simple observation, which can be seen as combination of \cite[Remark 2.12 \& Theorem 2.15]{Hairer2017} or checked directly.
\begin{lemma}\label{average bounds}
In the notation above, for $f\in \mathcal D_{p,q}^\gamma$ the following bound holds
\begin{equation}
\bigg(\sum_{n\geq 0} \sum_{h\in \Lambda^C_{n+1}} \bigg\| \frac{\big|\bar{f}^{n} (x) - \Gamma_{x,x+h} \bar{f}^{n+1}(x+h)\big|_\zeta}{2^{-n(\gamma-\zeta)}}\bigg\|_{l^p_n} \bigg)^{\frac1{q}} \lesssim \|f\|_{\mathcal{D}^\gamma_{p,q}} (1+\|\Gamma\|_\gamma)\, ,
\end{equation}
and when working with a second model and $g\in \tilde {\mathcal D}_{p,q}^\gamma$ we have
\begin{align}
\bigg(\sum_{n\geq 0} \sum_{h\in \Lambda^C_{n+1}} \bigg\| &\frac{\big|\bar{f}^{n} (x) - \Gamma_{x,x+h} \bar{f}^{n+1}(x+h)-\big(\bar{g}^{n} (x) - \tilde\Gamma_{x,x+h} \bar{g}^{n+1}(x+h)\big)\big|_\zeta}{2^{-n(\gamma-\zeta)}}\bigg\|_{l^p_n} \bigg)^{\frac1{q}}\\
& \lesssim \|f,g\|_{\mathcal{D}^\gamma_{p,q}} (1+\|\Gamma\|_\gamma)+\|\Gamma-\tilde\Gamma \|_\gamma \|g\|_{\mathcal{D}^\gamma_{p,q}} \, ,
\end{align}
where the implicit constants only depends on $\gamma$ and the regularity structure.
\end{lemma}
\begin{proof}[Proof of Prop. \ref{prop_reconstruction_for_smooth_models} for $p=q=\infty$ \& mollified models]
Since we assume our model is mollified, there exists some $\lambda\in (0,1)$ and $\phi\in C_c^\infty$ such that it is obtained by mollification with $\phi^\lambda$. Whenever there is a $\lambda$ appearing in the consequent arguments this one is meant. We let $F(y):=\big(\Pi_yf(y)\big)(y)$ for $ y \in \mathbb{R}^d$, and note that
$$\big(\Pi_xf(x)-F\big)(y)=\Pi_x\big(f(x)-\Gamma_{x,y}f(y)\big)(y) \, .$$

We start with the observation that for $y,z\in \mathbb{R}$ satisfying $|x-y|,|x-z|<\delta$ we have as a consequence of \eqref{pointwise good estimate} the following bound:
\begin{align*}
|\Pi_x\big(f(x)-\Gamma_{x,y}f(y)\big)(z)|&\leq \sum_{\zeta<\gamma} |\Pi_x Q_\zeta\big(f(x)-\Gamma_{x,y}f(y)\big)(z)|\\
&\lesssim \|\Pi\|_\gamma\sum_{\zeta<\gamma} |x-z|_{\mathfrak{s}}^\zeta |f(x)-\Gamma_{x,y}f(y)|_\zeta \\ 
&\lesssim \|\Pi\|_\gamma\|f\|_{\mathcal{D}^\gamma_{\infty,\infty}} \sum_{\zeta<\gamma} |x-z|_{\mathfrak{s}}^\zeta |x-y|_{\mathfrak{s}}^{\gamma-\zeta} \, .
\end{align*}
For $|x-y|<\delta$ this implies 
\begin{equation}\label{easiest estimate}
|\big(\Pi_xf(x)-F\big)(y)|\lesssim \|\Pi\|_\gamma\|f\|_{\mathcal{D}^\gamma_{\infty,\infty}} |x-y|_{\mathfrak{s}}^\gamma,
\end{equation} which implies the Bound \eqref{reconstruction cond} for test functions $\eta_x^\delta$ with $\delta\leq \lambda$. 

Now we consider the case of $\lambda\leq \delta$, which is considerably closer to the proof of the proposition in the general case. 
\begin{remark}
The gist of the next argument is that while we can control the integrand of  $\int_y \big(F(y)-\Pi_xf(x)(y)\big) \eta_x(y) dy$ well when $x$ and $y$ are close this is not the case when they are not. We resolve this problem by comparing local averages on dyadic scales and bootstrapping our good bounds from one scale to the next by localizing appropriately. While this is close in spirit to the classical argument via wavelets, we point out that we only use the existence of a partition of unity and localizing simply means localizing in space (opposed to "space and frequency" as done via wavelets).
\end{remark}

Fix $N_0$ maximal such that, if we denote by $x_0=x_0(x)$ the closest point to $x$ in $\Lambda_{N_0}$, the inclusion
$B_\delta(x) \subset Q^{N_0}_{x_0(x)}$ 
holds uniformly over $x\in \mathbb{R}$. (Note that this implies $\delta\sim 2^{-N_0}$.)
Also fix $N_1$ such that $\lambda\in [2^{-N_1}, 2^{(-N_1+1)}]$. Note that for $y\in B_\delta(x)$ we have
\begin{align}
F(y)&-\Pi_xf(x)(y) 
=\sum_{k\in \Lambda_{N_1}\cap B^{N_0}_{x_0}} \mathbb{1}^{N_1}_k(y) \frac{1}{|B^{N_1}_k|}\int_{B_k^{N_1}}\big(\Pi_{y}f(y)-\Pi_{\tilde{y}}f(\tilde{y})\big)(y) d\tilde{y}\label{three summands}\\
&+\sum_{k\in \Lambda_{N_1}\cap B^{N_0}_{x_0}}  \mathbb{1}^{N_1}_k(y) \frac{1}{|B^{N_1}_k|}\int_{B_k^{N_1}}\Pi_{\tilde{y}}f(\tilde{y})(y) d\tilde{y}-{\mathbb{1}}^{N_0}_{x_0}(y)\frac{1}{|B^{N_0}_{x_0}|}\int_{B^{N_0}_{x_0}}\Pi_{\tilde{y}}f(\tilde{y})(y) d\tilde{y}\nonumber\\
&+\frac{1}{|B^{N_0}_{x_0}|}\int_{B^{N_0}_{x_0}}\big(\Pi_{\tilde{y}}f(\tilde{y})-\Pi_xf(x)\big)(y) d\tilde{y}\nonumber\\
&=\sum_{k\in \Lambda_{N_1}\cap B^{N_0}_{x_0}} \mathbb{1}^{N_1}_k(y) \frac{1}{|B^{N_1}_k|}\int_{B_k^{N_1}}\big(\Pi_{y}f(y)-\Pi_{\tilde{y}}f(\tilde{y})\big)(y) d\tilde{y}\nonumber\\
&+\sum_{k\in \Lambda_{N_1}\cap B^{N_0}_{x_0}}  \mathbb{1}^{N_1}_k(y)\Pi_k\bar{f}^{N_1}(k)(y)-\mathbb{1}^{N_0}_{x_0}(y) \Pi_x\bar{f}^{N_0}_x(y)\nonumber\\
&+\frac{1}{|B^{N_0}_{x_0}|}\int_{B^{N_0}_{x_0}}\big(\Pi_{\tilde{y}}f(\tilde{y})-\Pi_xf(x)\big)(y) d\tilde{y} \nonumber\\
&=\sum_{k\in \Lambda_{N_1}\cap B^{N_0}_{x_0}} \mathbb{1}^{N_1}_k(y) \frac{1}{|B^{N_1}_k|}\int_{B_k^{N_1}}\big(\Pi_{y}f(y)-\Pi_{\tilde{y}}f(\tilde{y})\big)(y) d\tilde{y}\label{summand 1}\\
&+\sum_{n=N_0}^{N_1-1} \Big(\sum_{k'\in \Lambda_{n+1}\cap B^{N_0}_{x_0}}  \mathbb{1}^{n+1}_k(y)\Pi_k\bar{f}^{n+1}(k')(y)-\sum_{k\in \Lambda_{n}\cap B^{N_0}_{x_0}}\mathbb{1}^{n}_{k}(y)\Pi_{k}\bar{f}^{n}(k)(y)\Big)\label{summand 2}\\
&+\frac{1}{|B^{N_0}_{x_0}|}\int_{B^{N_0}_{x_0}}\big(\Pi_{\tilde{y}}f(\tilde{y})-\Pi_xf(x)\big)(y) d\tilde{y} \ .\label{summand 3}
\end{align}

When testing against $\eta^\delta_x$ the first summand~\eqref{summand 1} can be bounded using the point-wise (local) Bound~\eqref{easiest estimate}, the last summand~\eqref{summand 3} can be bounded using:
\begin{align*}
&|\Pi_x\big(  \Gamma_{x,\tilde{y}}f(\tilde{y})-\Pi_xf(x)\big)(\eta_x^\delta)|  \\
\lesssim & \|\Pi\|_\gamma \sum_{\zeta<\gamma} |\Gamma_{x,\tilde{y}}f(\tilde{y})-f(x)|_\zeta\delta^\zeta \\
\lesssim & \|f\|_{\mathcal{D}^\gamma_{\infty,\infty}} \|\Pi\|_\gamma\sum_{\zeta<\gamma} |x-\tilde{y}|_\mathfrak{s}^{\gamma-\zeta}\delta^\zeta\lesssim\|f\|_{\mathcal{D}^\gamma_{\infty,\infty}}  \|\Pi\|_\gamma\delta^\gamma \, , 
\end{align*}
which follows from the definition of a model. We still have to bound the second summand. We note that for $N_0\leq n<N_1$.
\begin{align}
&\int \Big(\sum_{k'\in \Lambda_{n+1}\cap B_{N_0}(x)} \mathbb{1}_{B_{n+1}(k')}\Pi_{k'}\bar{f}^{n+1}({k'})(y)-
\sum_{{k}\in \Lambda_{n}\cap B_{N_0}(x)} \mathbb{1}_{B_{n}({k})} \Pi_z\bar{f}^n({k})(y)\Big) \eta_x^{\delta} (y)\, dy\nonumber\\
&=\sum_{{k'}\in \Lambda_{n+1}\cap B_{N_0}(x)} \Pi_{k'}\bar{f}^{n+1}({k'})(\mathbb{1}_{B_{n+1}({k'})}\eta^\delta_x)-
\sum_{k\in \Lambda_{n}\cap B_{N_0}(x)} \Pi_{k}\bar{f}^n(z')(\mathbb{1}_{B_{n}(k)}\eta^\delta_x)\nonumber\\
&=\sum_{{k}\in \Lambda_{n}\cap B_{N_0}(x)} \sum_{{k'}\in \Lambda_{n+1}\cap B_{N_0}(x)}\Big( \Pi_{k'}\bar{f}^{n+1}({k'})(\mathbb{1}_{B_{n}(k)}\mathbb{1}_{B_{n+1}({k'})}\eta^\delta_x)-
\Pi_{k}\bar{f}^n(k)(\mathbb{1}_{B_{n+1}({k'})}\mathbb{1}_{B_{n}(k)}\eta^\delta_x)\Big)\nonumber\\
&=\sum_{k\in \Lambda_{n}\cap B_{N_0}(x)} \sum_{{k'}\in \Lambda_{n+1}\cap B_{N_0}(x)} \Big( \Pi_{k'}\bar{f}^{n+1}({k'})-
\Pi_{k}\bar{f}^n(k)\Big)(\mathbb{1}_{B_{n}(k)}\mathbb{1}_{B_{n+1}({k'})}\eta^\delta_x)\nonumber\\
&=\sum_{k\in \Lambda_{n}\cap B_{N_0}(x)} \sum_{{k'}\in \Lambda_{n+1}\cap B_{N_0}(x)} \Big( \Pi_{k'}\bar{f}^{n+1}({k'})-
\Pi_{k}\bar{f}^n(k)\Big)(\mathbb{1}_{B_{n}(k)}\mathbb{1}_{B_{n+1}({k'})}\eta^\delta_x)\nonumber\\
&=\sum_{k\in \Lambda_{n}\cap B_{N_0}(x)} \sum_{{k'}\in \Lambda_{n+1}\cap B_{N_0}(x): |{k'}-k|_\infty\leq C 2^{-n}}\Big( \Pi_{k'}\bar{f}^{n+1}({k'})-
\Pi_{k}\bar{f}^n(k)\Big)(\mathbb{1}_{B_{n}(k)}\mathbb{1}_{B_{n+1}({k'})}\eta^\delta_x)\nonumber\\
&=\sum_{k\in \Lambda_{n}\cap B_{N_0}(x)} \sum_{{k'}\in \Lambda_{n+1}\cap B_{N_0}(x): |{k'}-k|_\infty\leq C 2^{-n}} \Pi_{k}\Big(\Gamma_{k,{k'}}\bar{f}^{n+1}({k'})-
\bar{f}^n(k)\Big)(\mathbb{1}_{B_{n}(k)}\mathbb{1}_{B_{n+1}({k'})}\eta^\delta_x) \nonumber\\ 
&=\sum_{k\in \Lambda_{n}\cap B_{N_0}(x)} \sum_{h\in \Lambda_{n+1}^C} \Pi_{k}\Big(\Gamma_{k,k+h}\bar{f}^{n+1}(k+h)-
\bar{f}^n(k)\Big)(\mathbb{1}_{B_{n}(k)}\mathbb{1}_{B_{n+1}(k+h)}\eta^\delta_x) \label{difference over diadics}\, .
\end{align}
Since $\frac{\delta^d}{2^{-nd}}\mathbb{1}_{B_{n}(z')}\mathbb{1}_{B_{n+1}(z)}\eta^\delta_x=\Psi^{2^{-n}}_z$ for some $\Psi\in \mathcal{B}^r$, we can bound
\begin{align}
\Big|&\sum_{k\in \Lambda_{n}\cap B_{N_0}(x)} \sum_{h\in \Lambda_{n+1}^C} \Pi_{k}\Big(\Gamma_{k,k+h}\bar{f}^{n+1}(k+h)-
\bar{f}^n(k)\Big)(\mathbb{1}_{B_{n}(k)}\mathbb{1}_{B_{n+1}(k+h)}\eta^\delta_x) \Big|\nonumber\\
\leq& \|\Pi\|_\gamma \sum_{\alpha<\gamma}\frac{2^{-n|\mathfrak{s}|}}{2^{-N_0|\mathfrak{s}|}}\sum_{k\in \Lambda_{n}\cap B_{N_0}(x)} \sum_{h\in \Lambda_{n+1}^C}|\Gamma_{k,k+h}\bar{f}^{n+1}(k+h)-
\bar{f}^n(k)|_\alpha 2^{-\alpha n}\label{bound on difference over diadic}
\end{align}
Using Lemma~\ref{average bounds} for the case $p=q=\infty$ this is bounded by
\begin{align*}
\sum_{\alpha<\gamma}\frac{2^{-n|\mathfrak{s}|}}{2^{-N_0|\mathfrak{s}|}}&\sum_{k\in \Lambda_{n}\cap B_{N_0}(x)} \sum_{h\in \Lambda_{n+1}^C}|\Gamma_{k,k+h}\bar{f}^{n+1}(k+h)-
\bar{f}^n(k)|_\alpha 2^{-\alpha n}\\
&\lesssim \|f\|_{\mathcal{D}^\gamma_{\infty,\infty}}(1+\|\Gamma\|_\gamma) \sum_{\alpha<\gamma}\frac{2^{-n|\mathfrak{s}|}}{2^{-N_0|\mathfrak{s}|}}\sum_{z'\in \Lambda_{n}\cap B_{N_0}(x)}  2^{-n(\gamma-\alpha)} 2^{-\alpha n}\\
&=\|f\|_{\mathcal{D}^\gamma_{\infty,\infty}}(1+\|\Gamma\|_\gamma) \sum_{\alpha<\gamma}\frac{2^{-n|\mathfrak{s}|}}{2^{-N_0|\mathfrak{s}|}}\sum_{z'\in \Lambda_{n}\cap B_{N_0}(x)}  2^{-n\gamma}\\
&\lesssim \|f\|_{\mathcal{D}^\gamma_{\infty,\infty}}(1+\|\Gamma\|_\gamma) \sum_{\alpha<\gamma} 2^{-n\gamma} \, .
\end{align*}
Thus we conclude with 
\begin{align}\label{concluding estimate}
|\langle F(y)-\Pi_xf(x),\eta^\delta\rangle| \lesssim \|f\|_{\mathcal{D}^\gamma_{\infty,\infty}}  \|\Pi\|_\gamma(1+\|\Gamma\|_\gamma) (\delta^\gamma +\sum_{N_0}^{N_1} 2^{-n\gamma})\lesssim \|f\|_{\mathcal{D}^\gamma_{\infty,\infty}}  \|\Pi\|_\gamma(1+\|\Gamma\|_\gamma) \delta^\gamma \, .
\end{align}
Thus we conclude $F=\mathcal{R}f$, and the reconstruction bound \eqref{reconstruction cond}.
The continuity bound \eqref{continuity bound} can be obtained in an analogous manner, applying the same arguments to
\begin{align}
&\Pi_x f(x)(y)- F(y)-\Big( \bar{\Pi}_x\bar f(x)(y)-\bar{F}(y)\Big) \nonumber\\
&=\Pi_xf(x)(y)-\Pi_y f(y)(y)-\Big(\bar{\Pi}_x\bar f(x)(y)-\bar\Pi_y\bar f(y)(y)\Big)\nonumber\\
&=\Pi_x\big({f}(x)-\Gamma_{x,y}{f}(y)-\bar f(x)+\bar\Gamma_{x,y}\bar{f}(y) \big)(y)+\big(\Pi_x-\bar \Pi_x\big)\big(\bar\Gamma_{x,y}\bar{f}(y)-\bar{f}(x)\big)(y)\label{difference}\, .
\end{align}
\end{proof}
\begin{remark} Before we proceed to the general case, let us observe that the argument above works for $p<\infty$, but as soon as $q<\infty$ it does not work any more, since our ``small scale bound'' \eqref{easiest estimate} is weaker than the required bound \eqref{reconstruction cond}. Let us, however, note that the argument for larger scales did not crucially depend on $N_1$ as witnessed by \eqref{concluding estimate} and thus we could by letting all scales be ``large'', i.e.~by letting $N_1\to \infty$, modify our strategy. This is done next.
\end{remark}
\begin{proof}[General case of Prop. \ref{prop_reconstruction_for_smooth_models}]
We introduce another piece of notation from \cite{Hairer2017}: We write for $n_0\in\mathbb{N}$ and $f\in L_\delta^{q}$ 
$$\|f\|_{L_{n_0}^q}:=\int_{2^{-n_0-1}}^{2^{-n_0}} |f(\delta)|^q|\frac{d\delta}{\delta}.$$
and note that we have the identity $\|f\|_{L^q_\delta}=\| \| f\|_{L^q_{n_0}}\|_{l^q(n_0)}$.

Let again $N_0=N_0(\delta)$ be the largest integer such that, if we denote by $x_0=x_0(x)$ be the closest point to $x$ in $\Lambda_{N_0}$, the inclusion
$B_\delta (x) \subset B^{N_0}_{x_0(x)}$ 
holds uniformly over $x\in \mathbb{R}$. Recall that this implies $\delta\sim 2^{-N_0}$ and in particular for $\delta \in [2^{-(n_0+1)},2^{-n_0}]$ this implies $2^{-N_0}\sim 2^{-n_0}$.

Recalling Equation~\eqref{three summands} we can write\footnote{Note that for $\delta\sim 2^{-n_0}$ this is reminiscent of the identity $$\big(\mathcal{R}f-\Pi_xf(x)\big)(\eta^\delta_x)=\big(\mathcal{R}_{n_0} f-\mathcal{P}_{n_0}\Pi_xf(x)\big)(\eta^\delta_x)+\sum_{n\geq n_0} \big(\mathcal{R}_{n+1}f-\mathcal{R}_{n}f-\mathcal{P}_n^\perp\Pi_xf(x)\big)(\eta^\delta_x)$$ in \cite[Section 3.2]{Hairer2017}.}
\begin{align}
\langle F-\Pi_xf(x), \eta^\delta \rangle&= \sum_{n=N_0}^{\infty} \Big(\sum_{k'\in \Lambda_{n+1}\cap B^{N_0}_{x_0}}  \mathbb{1}^{n+1}_k(\cdot)\Pi_{k'}\bar{f}^{n+1}(k')-\sum_{k\in \Lambda_{n}\cap B^{N_0}_{x_0}}\mathbb{1}^{n}_{k}(\cdot)\Pi_{k}\bar{f}^{n}(k)\Big)(\eta^\delta)\\
&+\frac{1}{|B^{N_0}_{x_0}|}\int_{B^{N_0}_{x_0}}\big(\Pi_{\tilde{y}}f(\tilde{y})-\Pi_xf(x)\big)(\eta^\delta) d\tilde{y} ,
\end{align}
which we obtain by letting $N_1\to \infty$, since then the first summand \eqref{summand 1} converges to $0$ as we are working with smooth models. 

Now we estimate the first summand
\begin{align*}
&\Big\| \big\| \sup_\eta \frac{\frac{1}{|B^{N_0}_{x_0}|}\int_{B^{N_0}_{x_0}}\big(\Pi_{\tilde{y}}f(\tilde{y})-\Pi_xf(x)\big)(\eta_x^\delta) d\tilde{y}}{\delta^\gamma} \big\|_{L^p(dx)}\Big\|_{L^q_{n_0}(d\delta_x)}\\
&=\Big\| \big\| \sup_\eta \frac{\frac{1}{|B^{N_0}_{x_0}|}\int_{B^{N_0}_{x_0}}\Pi_x\big(\Gamma_{x,\tilde{y}}f(\tilde{y})-f(x)\big)(\eta_x^\delta) d\tilde{y}}{\delta^\gamma} \big\|_{L^p(dx)}\Big\|_{L^q_{n_0}(d\delta_x)}\\
&=\Big\| \big\| \sup_\eta \frac{\frac{1}{|B^{N_0}_{0}|}\int_{B^{N_0}_{0}}\Pi_x\big(\Gamma_{x,x+h}f(x+h)-f(x)\big)(\eta_x^\delta) d\tilde{h}}{\delta^\gamma} \big\|_{L^p(dx)}\Big\|_{L^q_{n_0}(d\delta_x)}\\
&\leq\ \|\Pi\|_\gamma \Big\| \big\|  \frac{1}{|B^{N_0}_{0}|}\int_{B^{N_0}_{0}}\sum_{\zeta<\gamma}\frac{\big|\Gamma_{x,x+h}f(x+h)-f(x)\big|_{\zeta}}{|h|_\mathfrak{s}^{\gamma-\zeta}} d\tilde{h} \big\|_{L^p(dx)}\Big\|_{L^q_{n_0}(d\delta_x)}\\
&\leq\|\Pi\|_\gamma\Big\|   \frac{1}{|B^{N_0}_{0}|}\int_{B^{N_0}_{0}}\sum_{\zeta<\gamma}\frac{\big\||\Gamma_{x,x+h}f(x+h)-f(x)\big|_{\zeta}\|_{L^p(dx)}}{|h|_\mathfrak{s}^{\gamma-\zeta}} d\tilde{h} \Big\|_{L^q_{n_0}(d\lambda_x)}\\
\end{align*}
Note that since for each $n_0$ we have $n_0\sim \delta \sim N_0$ we can take the $l^q$ norm and use Jensen's inequality to obtain the desired bound.

Note that we can rewrite and bound the first summand using \eqref{difference over diadics} and \eqref{bound on difference over diadic}
\begin{align*}
&|\sum_{n=N_0}^{\infty} \Big(\sum_{k'\in \Lambda_{n+1}\cap B^{N_0}_{x_0}}  \mathbb{1}^{n+1}_{k'}(\cdot)\Pi_{k'}\bar{f}^{n+1}(k')-\sum_{k\in \Lambda_{n}\cap B^{N_0}_{x_0}}\mathbb{1}^{n}_{k}(\cdot)\Pi_{k}\bar{f}^{n}(k)\Big)(\eta_x^\delta)|\\
&\leq\|\Pi\|_\gamma\sum_{n=N_0}^{\infty} \frac{2^{-nd}}{\delta^d}\sum_{k\in \Lambda_{n}\cap B_{N_0}(x)} \sum_{h\in \Lambda_{n+1}^C} \sum_{\zeta<\gamma} |\Gamma_{k,k+h}\bar{f}^{n+1}(k+h)-
\bar{f}^n(k)|_\zeta 2^{-n\zeta}
\end{align*}
The expression obtained by dividing by $\delta^\gamma$, taking the supremum over $\eta\in \mathcal{B}_r$, then taking the $L^p(dx)$ and $\L^q_{n_0}$ norm is thus given by 
\begin{align*}
&\|\Pi\|_\gamma\Big\| \big\|\sum_{n=N_0}^{\infty} \frac{2^{-nd}}{\delta^d}\sum_{k\in \Lambda_{n}\cap B_{N_0}(x)} \sum_{h\in \Lambda_{n+1}^C} \sum_{\zeta<\gamma} \frac{|\Gamma_{k,k+h}\bar{f}^{n+1}(k+h)-
\bar{f}^n(k)|_\zeta 2^{-n\zeta}}{\delta^\gamma}\big\|_{L^p(dx)}\Big\|_{L^q_{n_0}}\\
&\lesssim \|\Pi\|_\gamma\Big\| \big\|\sum_{n=N_0}^{\infty} \frac{2^{-nd}}{\delta^d}\sum_{k\in \Lambda_{n}\cap B_{N_0}(x)} \sum_{h\in \Lambda_{n+1}^C} \sum_{\zeta<\gamma} \frac{|\Gamma_{k,k+h}\bar{f}^{n+1}(k+h)-
\bar{f}^n(k)|_\zeta 2^{-n\zeta}}{2^{-n_0\gamma}}\big\|_{L^p(dx)}\Big\|_{L^q_{n_0}}\\
&\lesssim \|\Pi\|_\gamma\Big\| \big\|\sum_{n=N_0}^{\infty} \frac{2^{-nd}}{\delta^d}\sum_{k\in \Lambda_{n}\cap B_{N_0}(0)} \sum_{h\in \Lambda_{n+1}^C} \sum_{\zeta<\gamma} \frac{2^{(n_0-n)\gamma}|\Gamma_{x+k,x+k+h}\bar{f}^{n+1}(x+k+h)-
\bar{f}^n(x+k)|_\zeta }{2^{-n(\gamma-\zeta)}}\big\|_{L^p(dx)}\Big\|_{L^q_{n_0}}\\
&\lesssim \|\Pi\|_\gamma\Big\| \sum_{n=N_0}^{\infty} 2^{(n_0-n)\gamma}\sum_{h\in \Lambda_{n+1}^C} \sum_{\zeta<\gamma} \frac{\big\||\Gamma_{k,k+h}\bar{f}^{n+1}(k+h)-
\bar{f}^n(k)|_\zeta\big\|_{l_n^p(k)} }{2^{-n(\gamma-\zeta)}}\Big\|_{L^q_{n_0}}\ .
\end{align*}
Now using Jensen's inequality when taking the $l^q$ norm and applying Lemma \ref{average bounds} yields the desired bound \eqref{reconstruction cond}. The bound for two different models and modelled distributions \eqref{continuity bound} can be obtained by using the same arguments on \eqref{difference}.
The simple uniqueness argument is left to the reader\footnote{Alternativelly see Step 5 in the proof of \cite[Theorem 3.1]{Hairer2017}. }. 
\end{proof}

\subsection{The reconstruction theorem by extension from smooth models}

We shall now proof the reconstruction theorem in full generality. Let us start by recording the following 
\begin{prop}\label{last prop}
Let $\mathcal{T}=(A,T,G)$ be a regularity structure and $Z=(\Pi,\Gamma),\bar{Z}=(\bar{\Pi},\bar{\Gamma})\in \mathcal{M}_\mathcal{T}$ smooth models. Then the reconstruction operator satisfies the following inequalities for sufficiently small $\epsilon>0$.
\begin{equation}\label{reconstruction cond,epsilon}
\Bigg\| \Big\| \sup_{\eta\in \mathcal{B}^r} \frac{|\langle \mathcal{R}_Z f-\Pi_x f(x),\eta_x^{\delta}\rangle|}{\delta^{\gamma-\epsilon}} \Big\|_{L^p}\Bigg\|_{L^q_\delta} \lesssim \|f\|_{\mathcal{D}^{\gamma}_{p,q}} \|\Pi\|_{\gamma}(1+\|\Gamma\|_\gamma) 
\end{equation}
\begin{align}
\Bigg\| \Big\| \sup_{\eta\in \mathcal{B}^r} &\frac{|\langle \mathcal{R}_Z f-\Pi_x \bar{f}(x)-\mathcal{R}_{\bar{Z}}\bar{f}-\bar{\Pi}_x f(x),\eta_x^\delta\rangle|}{\delta^{\gamma-\epsilon}} \Big\|_{L^p}\Bigg\|_{L^q_\delta} \nonumber\\
\lesssim & \|f,\bar{f}\|_{\mathcal{D}^{\gamma}_{p,q}} \|\Pi\|_\gamma 
+\|\bar{f}\|_{\mathcal{D}^\gamma_{p,q}} \big(\|\Pi-\bar{\Pi}\|_\gamma (1+\|\Gamma\|_\gamma)+
\|\bar{\Pi}\|_\gamma \|\Gamma-\bar{\Gamma}\|_\gamma\big)\label{continuity bound,epsilon}
\end{align}
and
\begin{align}
\Bigg\| \Big\| \sup_{\eta\in \mathcal{B}^r} &\frac{|\langle \mathcal{R}_Z f-\Pi_x \bar{f}(x)-\mathcal{R}_{\bar{Z}}\bar{f}-\bar{\Pi}_x f(x),\eta_x^\delta\rangle|}{\delta^{\gamma-\epsilon}} \Big\|_{L^p}\Bigg\|_{L^q_\delta} \nonumber\\
\lesssim & \|f,\bar{f}\|_{\mathcal{D}^{\gamma-\epsilon}_{p,q}} \|\Pi\|_\gamma 
+\|\bar{f}\|_{\mathcal{D}^\gamma_{p,q}} \big(\|\Pi-\bar{\Pi}\|_{\gamma,\epsilon} (1+\|\Gamma\|_\gamma)+
\|\bar{\Pi}\|_\gamma \|\Gamma-\bar{\Gamma}\|_{\gamma,\epsilon}\big) \label{continuity bound,epsilon 2}
\end{align}
uniformly over all $f\in \mathcal{D}^\gamma_{p,q}, \bar{f}\in \bar{\mathcal{D}}^{\gamma}_{p,q}$ and all models $Z,\bar{Z}\in\mathcal{M}_\mathcal{T}$.
\end{prop}

\begin{proof}
Estimates \eqref{reconstruction cond,epsilon} and \eqref{continuity bound,epsilon} are obvious consequences of Proposition \ref{prop_reconstruction_for_smooth_models}. Estimate \eqref{continuity bound,epsilon 2} follows for $\epsilon< \text{dist}(\gamma, A\setminus \{\gamma\})$ by an a verbatim adaptation of the proof of the same Proposition.
\end{proof}

Now the following "soft" argument based on Proposition \ref{last prop} concludes the proof of the reconstruction theorem in full generality.
\begin{proof}
Let $Z$ be a model and $ f \in \mathcal{D}^\gamma_{p,q} $ a modelled distribution, then we can construct families $ Z^\delta $ of smooth models by Theorem~\ref{thm_approximating_smooth_models} and a family $ f^\lambda $ of modelled distributions with respect to $ Z^\lambda $ such that
\[
{\| Z^\lambda - Z \|}_{\gamma,\epsilon} \lesssim \lambda^\epsilon  
\]
and
\[
{\| f^\lambda - f \|}_{\mathcal{D}^\gamma_{p,q}} \rightarrow 0
\]
as $ \lambda \to 0 $. The second estimate follows from Lemma \ref{lem_density} by approximating $ f $ with elements from $ \mathcal{E}^\gamma_{p,q} $ and then by changing the model from $ Z $ to $ Z^\lambda $. Both estimates together with Equation \eqref{continuity bound,epsilon 2} yield Equation\eqref{reconstruction cond,epsilon} for the (not necessarily smooth) model $Z$. Now letting $\epsilon\to 0$ gives the reconstruction bound \eqref{reconstruction cond} of Theorem~\ref{thm_reconstruction}. The Bound \eqref{continuity bound} follows similarly, by using \eqref{continuity bound,epsilon} instead of \eqref{reconstruction cond,epsilon}.
\end{proof}

\begin{remark}
At this point it appears that several deep facts of regularity structures lie in its splendid definitions which allow in a seamless way to translate "obvious" statements for smooth models to the closure of smooth models with respect to a well chosen topology.
\end{remark}

\section{Consequences for Rough Path theory}
In this section we illustrate the above constructions in the setting of branched rough path theory. Of course an analogous discussion applies for geometric rough paths. We shall use the notation in \cite{HaiKel:15}.
\subsection{Recall on branched rough path theory}
Denoting by $\mathcal{H}$ the Connes Kreimer algebra with nodes indexed by $\{1,...,d\}$ and for $\tau \in \mathcal{H}$ by $|\tau|$ the number of nodes of $\tau$. 

The following is essentially \cite[Definition 2.13]{HaiKel:15}.
\begin{definition}\label{def branched rough path}
For $\gamma\in (0,1)$ fix $N$ to be the largest integer, such $\gamma N\leq 1$. A $\gamma$-H\"older branched rough path is a map $\mathbf{X}:\mathbb{R}\to G_N(\mathcal{H})$ satisfying the bound
\begin{equation}\label{rough path bound}
\sup_{|t-s|<1}\frac{|\langle \mathbf{X}_{s,t},\tau\rangle|}{|t-s|^{\gamma |\tau|}}<+\infty
\end{equation}
for every forest $\tau\in \mathcal{F}_N$, where we write $\mathbf{X}_{s,t}:=\mathbf{X}_{s}^{-1}\star\mathbf{X}_{t}$ and $|\tau|$ denotes the number of nodes of $\tau$. 
We call a branched rough path $\mathbf{X}:\mathbb{R}\to G_N(\mathcal{H})$ normalized, if  $\mathbf{X}_0=\epsilon$ is the unit in $G_N$. We denote by $\mathcal{C}_{br}^{\gamma}$ the space of all normalized branched rough paths and equip it with the topology induced by the (inhomogeneous) metric 
\begin{equation}\label{rough path metric}
\| \mathbf{X}-\mathbf{X}\|=\sup_{\tau\in \mathcal{F}_N}\sup_{|s-t|< 1}\frac{|\langle \mathbf{X}_{s,t}-\mathbf{X}_{s,t},\tau\rangle|}{|t-s|^{\gamma |\tau|}}.
\end{equation}
For a path $X:\mathbb{R}\to \mathbb{R}^d$, we say $\mathbf{X}$ is a branched rough path above $X$ if $\langle \mathbf{X}_{s,t}, [1]_j\rangle=X_j(t)-X_j(s)$ for all $j\in \{1,...,d\}.$
\end{definition}
The following is essentially \cite[Definition 3.2]{HaiKel:15}
\begin{definition}\label{def controlled rough path}
Suppose $\mathbf{X}$ is $\gamma$-H\"older rough path and let $N\in \mathbb{N}$ be the largest integer such that $N\gamma\leq 1$. An $\mathbf{X}$-controlled rough path is a bounded path $\mathbf{Z}:\mathbb{R}\to \mathcal{H}_{N-1}$ satisfying
$$\sup_{|t-s|<1}\frac{|\langle \tau,\mathbf{Z}_t\rangle-\langle \mathbf{X}_{s,t} \star \tau,\mathbf{Z}_s\rangle |}{|t-s|^{(N-|h|)\gamma}} <+\infty$$
for all forests $\tau\in \mathcal{F}_{N-1}$. 

We equip the space of all $\mathbf{X}$-controlled rough paths with the norm
$$\|\mathbf{Z}\|= \sup_{\tau\in \mathcal{F}_{N-1}} \sup_{t\in \mathbb{R}} |\langle \tau, \mathbf{Z}_t\rangle+\sup_{s,t\in\mathbb{R}:|t-s|<1}\frac{|\langle \tau,\mathbf{Z}_t\rangle-\langle \mathbf{X}_{s,t} \star \tau,\mathbf{Z}_s\rangle |}{|t-s|^{(N-|h|)\gamma}}. $$
For a path $Z:\mathbb{R}\to \mathbb{R}$, we call $\mathbf{Z}$ a controlled path above $Z$, if $\langle 1,\mathbf{Z}\rangle=Z$.
\end{definition}

Lastly let us recall that the main theorem, which makes (branched) rough path theory functional is the existence of the rough integral. We can construct the rough integral for a branched rough paths $\mathbf{X}\in\mathcal{C}^\gamma_{br}$ and an $\mathbf{X}$-controlled rough path $\mathbf{Z}$ the rough integral can be defined as 
\begin{equation}\label{equation rough integral}
\big(\int_s^t Z_u \, d\mathbf{X}_u\big)_j:=\lim_{|\mathcal{P}(s,t)|\to 0} \sum_{[u,v]\in\mathcal{P}^n(s,t)} \tilde{Z}^j_{u,v} \, ,
\end{equation}
where 
\begin{equation}\label{equation rough integral2}
\tilde{Z}^j_{u,v}:=\sum_{\tau\in \mathcal{F}_{N-1}} \langle \tau, \mathbf{Z}_s\rangle\langle \mathbf{X}_{s,t},[\tau]_j\rangle.
\end{equation}  We refer to \cite[Section 3.1]{HaiKel:15} and \cite{Gubinelli2011} for more details.

\subsection{The branched rough path regularity structure}
Next we recall the construction of the branched rough path regularity structure, see~\cite[Section 4.4]{Hairer2014} for similar considerations.
\begin{enumerate}
\item We let $A^{br}=\alpha\cdot \{0,1,...,N\}$ and declare a forest $\tau\in \mathcal{H}$ with $n$ nodes to have homogeneity $\alpha n$. 
\item Then we set $T^{br}=\mathcal{H}_N$.
\item The structure group $G^{br}$ is given by all linear maps which can be written as $(g^{-1}\otimes \id)\triangle$ for some $g$ a multiplicative functional on $\mathcal{H}$, where $\triangle$ denotes the Connes Kreimer coproduct.
\end{enumerate}
One has the following well known result, which is "essentially" contained in \cite[Section 4.4]{Hairer2014}. See also the master thesis of the first author for a detailed proof.

\begin{theorem}\label{controlled paths are modelled dist}
For the choice $\alpha=\gamma\in (0,1)$, the space $\mathcal{M}^1_{br}$ of models for $\mathcal{T}^{br}=(A^{br},T^{br},G^{br})$ satisfying $\Pi \mathbf{1}=1$ is isomorphic to the space of normalized branched rough paths $\mathcal{C}^\gamma_{br}$ in the sense that the following map is bijective and bi-continuous:
$$\mathcal{C}^\gamma_{br} \to \mathcal{M}^1_{br},\ \ \mathbf{X}\mapsto Z^{\mathbf{X}}=(\Pi^\mathbf{X},\Gamma^\mathbf{X})$$
where $$\Pi_s^\mathbf{X}(\tau)(t):=\langle\mathbf{X}_{s,t},\tau\rangle,$$ $$\Gamma_{s,t}^\mathbf{X}\tau:=(\mathbf{X}_{t,s}\otimes \id)\triangle\tau.$$

Furthermore, if we fix a rough path $\mathbf{X}$ and model $Z=(\Pi^\mathbf{X},\Gamma^\mathbf{X})$, then the space of $\mathbf{X}$-controlled paths is isomorphic to the space $\mathcal{D}_{\infty,\infty}^{N\gamma}$ for $N=\max \{n\in \mathbb{N} \ | \ n\gamma\leq 1\}$. The isomorphism of normed vector spaces is given by the identity map.
\end{theorem}

This theorem can now be used to translate all the results of this paper into the rough path setting, we illustrate this with some examples\footnote{Note that the case $N\gamma=1$ is not treated due to Assumption \ref{ass_polynomial_reg_structure}, though this restriction can be circumvented.}
\begin{enumerate}
\item\label{1} Using the approximation procedure in Section \ref{subsec mollifying} yields the following non-standard smooth approximation for rough paths:
\begin{align*}\label{mollified rough path}
\langle \mathbf{X}^\epsilon_{s,t}, \tau\rangle= \big( \phi^\epsilon*\langle \mathbf{X}_{s,\cdot} , \tau \rangle\big)(t)-\big( \phi^\epsilon*\langle \mathbf{X}_{s,\cdot} , \tau \rangle\big)(s)=\big( \phi^\epsilon*\langle \mathbf{X}_{s,\cdot} , \tau \rangle\big)_{s,t}\ , 
\end{align*}
for non empty trees $\tau$.
It is straight-forward to reformulate the topologies introduced for models in Section \ref{subsec topologies} to topologies for rough paths in order to obtain appropriate convergence results directly.
\item\label{2} Given an $\mathbf{X}$-controlled rough path $\mathbf{Z}$, the mollification of modelled distributions introduced in Remark \ref{Remakr mollifying modelled} gives the following approximating sequence of $\mathbf{X}^\epsilon$-controlled rough paths
$$\mathbf{Z}^\epsilon(t)= \mathbf{Z}(t)+\mathbf{1}\big(\phi^\epsilon\star Z\big)(t) \ ,$$
where $\mathbf{1}$ denotes the empty forest.
\item\label{3} The density statement in Lemma~\ref{lem_density} gives a dense subset of the space of controlled rough paths "independent" of the rough path they are controlled by.
\item If we apply the first two observations \eqref{1} and \eqref{2} in combination with the definition of the rough integral \eqref{equation rough integral} one obtains the following nonstandard approximation for the rough integral. 
$$\big(\int_s^t \mathbf{Z} \, d\mathbf{X}^\epsilon\big)_j=\lim_{\epsilon\to 0} \big(\int_s^t \mathbf{Z}^\epsilon \, d\mathbf{X}^\epsilon\big)_j\ ,$$
where 
\begin{align*}
\big(\int_s^t \mathbf{Z}^\epsilon \, d\mathbf{X}^\epsilon\big)_j:=
&\lim_{|\mathcal{P}(s,t)|\to 0} \sum_{[u,v]\in\mathcal{P}^n(s,t)} \sum_{\tau\in \mathcal{F}_{N-1}}\langle \tau, \mathbf{Z}_u^\epsilon\rangle\langle \mathbf{X}^\epsilon_{u,v}[\tau]_j\rangle\ .
\end{align*}
We know from general principles that the latter converges to the process
$$\int \phi^\epsilon(\delta) \big(\int_{s-\delta}^{t-\delta} \mathbf{Z}_u \, d\mathbf{X}_u\big)_j d\delta \, .$$
\item Alternatively one could of course conbine \eqref{1}, \eqref{3} and \eqref{equation rough integral} to construct other nonstandard approximations.

\item Recall the fact that the reconstruction theorem gives an easy proof of the existence of the rough integral, cf ~\cite[Lemma 3.2]{Hai14a}. By mimicking this construction for smooth paths and translating our proof of Theorem \ref{thm_reconstruction} into that framework one obtains a different proof of the existence of the rough integral.
\end{enumerate}

\newpage

\section{Appendix}
\subsection{Proof of Lemma \ref{lem_convergence_mollified_models}}
\begin{proof}
Note that we have to bound only $\|\Pi-\Pi^{\lambda}\|_{x,\zeta-\epsilon}$ and $\|\Gamma-\Gamma^{\lambda}\|_{x,y,\zeta-\epsilon}$, since the expressions in $\|Z,Z^\lambda\|_{\gamma,\epsilon}$ corresponding to monomials (i.e.~$\zeta\in \mathbb{N}$) vanish by definition. We start with the necessary bound on $$\|\Pi-\Pi^{\lambda}\|_{x,\zeta-\epsilon}:=\sup_{\varphi\in \mathfrak{B}^r} \sup_{\tau \in T_\zeta, |\tau|_\zeta=1 } \sup_{\delta<1} \frac{|\langle  \Pi_x-\Pi_x^{\lambda}\tau, \varphi^\delta_x \rangle |}{ \delta^{\zeta-\epsilon}} \, .$$
We bound the expression $\frac{|\langle  \Pi_x-\Pi_x^{\lambda}\tau, \varphi^\delta_x \rangle |}{ \delta^{\zeta-\epsilon}}$ by distinguishing between two cases:
\begin{enumerate}
\item If $\delta<\lambda$ we easily estimate:
$$\frac{|\langle  \Pi_x-\Pi_x^{\lambda}\tau, \varphi^\delta_x \rangle |}{ \delta^{\zeta-\epsilon}}\leq \|\Pi-\Pi^\lambda\|_{x,\zeta} \delta^\epsilon <\|\Pi-\Pi^\lambda\|_{x,\zeta} \lambda^\epsilon \, .$$
\item The case $\delta\geq\lambda$ needs slightly more work:
\begin{itemize}
\item We start with the case $\tau\in T_\zeta$ for  $\zeta<0$: we estimate each term of 
$$|(\Pi_x-\Pi^{\lambda}_x)(\varphi^{\delta})| \leq \sum_{n=0}^\infty |(\Pi_x^{\lambda/2^{n+1}}-\Pi^{{\lambda/2^n}}_x)(\varphi^{\delta})|$$ separately. Note that 
\begin{align*}
\big(\Pi_x^{\lambda/2^{n+1}}-\Pi^{{\lambda/2^n}}_x\big)(\varphi^{\delta})=&(\Pi_x)(\phi^{\lambda/2^{n+1}}-\phi^{\lambda/2^n})*(\varphi^{\delta})\\
=&\big(\varphi^{\delta}*\Pi_x\big)(\phi_x^{\lambda/2^{n+1}}-\phi_x^{\lambda/2^n}) \, .
\end{align*}
Now we use \eqref{pointwise estimate} to bound
$$|(\Pi_x-\Pi^{\lambda}_x)(\varphi^{\delta})|\lesssim \|\Pi_x\|_\zeta \sum_{\alpha<\zeta} \delta^{\alpha}\big(\frac{\lambda}{2^n}\big)^{\zeta-\alpha}\, .$$
Summing over $n$ and dividing by $\delta^{\zeta-\epsilon}$ we obtain:
$$\frac{|(\Pi_x-\Pi^{\lambda}\tau_x)(\varphi^{\delta})|}{\delta^{\zeta-\epsilon}} \lesssim \|\Pi_x\|_\zeta \sum_{\alpha<\zeta} \delta^{\alpha-\zeta+\epsilon}\lambda^{\zeta-\alpha}\lesssim \|\Pi\|_{x,\zeta} \lambda^\epsilon \, .$$
\item In the case $\zeta>0$ we write 
$$\Pi_x-\Pi_x^{\lambda}\tau =\Big(\Pi_x\tau-\phi^{\lambda}*\Pi_x\tau\big)+\Pi_x J(x)\tau \, .$$
The first summand can be estimated as in the case $\zeta<0$, while for the second summand we have:
$$|\Pi_x J(x)\tau(\varphi^\delta)|\lesssim  \|\Pi\|_{x,\zeta} \sum_{|k|<\zeta} \delta^{|k|} \lambda^{\zeta-|k|}\lesssim  \|\Pi\|_{x,\zeta} \delta^{\zeta-\epsilon}\lambda^\epsilon\, . $$
\end{itemize} 
\end{enumerate}
It remains to find an appropriate bound on $\|\Gamma-\Gamma^{\lambda}\|_{x,y,\zeta-\epsilon}$. Note that $\Gamma^{\lambda}-\Gamma= J(x)\Gamma_{x,y}-\Gamma_{x,y}J(y)$, which vanishes on $\tau\in T_{\zeta}$ for $\zeta<0$ and $\zeta\in\mathbb{N}$. Thus we have to bound 
$\frac{|\Gamma_{x,y}\tau|_m}{\|x-y\|_s^{\zeta-m}|\tau|}$
in the remaining case, where we again distinguish two cases:
\begin{enumerate}
\item the case $|x-y|<\lambda$: here it follows directly from \eqref{bound on Gamma} that
$$\frac{ | J(x)\Gamma_{x,y}-\Gamma_{x,y}J(y)|_m}{|x-y|^{\zeta-m-\epsilon}}\lesssim \|\Pi\|_{x,\zeta} |x-y|^\epsilon\lesssim \|\Pi\|_{x,\zeta} \lambda^\epsilon \, .$$
\item in the case $\lambda<|x-y|$ we bound $|J(x)\Gamma_{x,y}|_m$ and $|\Gamma_{x,y}J(y)|_m$ separately, which follows directly from Equation \eqref{second gamma bound}, respectively Equation \eqref{third gamma bound}.
\end{enumerate}
\end{proof}
\subsection{Proof of Lemma \ref{lem_density}}
Let $\mathbb{1}$ be a smooth function with compact support, such that ${(\mathbb{1}^n_k})_{k\in\Lambda_n}$,
where we have used the same notation as in Section \ref{subsec reconstruct smooth} \eqref{notation}, is a smooth partition of unity adapted to cubes $B^n_k$. 
\begin{proof}
We prove the claim in the case $p,q<\infty$. For $y\in\mathbb{R}^d$, denote by $\mathbf{1}_y^n$ the canonical lift of $\mathbb{1}^n_y$. Given modelled distribution $f \in \mathcal{D}^\gamma_{p,q} $ we approximate it locally by the modelled distributions
\[
x \mapsto f^n(x)=\int_y \mathbf{1}^n_y(x) \star\Gamma_{x,y} f(y) \, .
\]
Of course these modelled distributions can be approximated by "localized constants".
First we calculate:
\begin{align*}
Q_\alpha\big( f^n (x) -f(x)\big)=&Q_{\alpha} \int_{y} \mathbf{1}^n_y(x)\star\big(\Gamma_{x,y} f(y) -f(x)\big)\\
=&\int_{y} \sum_{ m+\beta=\alpha} Q_m \big(\mathbf{1}^n_y(x)\big) \star Q_\beta \big(\Gamma_{x,y} f(y) -f(x)\big),\\
=&\int_{y} \sum_{ m+\beta=\alpha} Q_m \big(\mathbf{1}^n_{x-y}(x)\big) \star Q_\beta \big(\Gamma_{x,x-y} f(x-y) -f(x)\big) \, ,
\end{align*}
where $m$ runs over $\mathbb{N}$ and $\beta$ over $A$. From this the estimate
\begin{align*}
\| |f^n (x) -f(x)|_\alpha \|_{L^p}&\leq \int_{y}  \sum_{ m+\beta=\alpha}  \| \big|\mathbf{1}^n_{x-y}(x)\big|_m \  \big|\Gamma_{x,x-y} f(x-y) -f(x)\big|_\beta\|_{\L^p(dx)}\\
&= \int_{y}  \sum_{ m+\beta=\alpha}  \| \big|\mathbf{1}^n_{x}(x+y)\big|_m \  \big|\Gamma_{x+y,x} f(x) -f(x+y)\big|_\beta\|_{\L^p(dx)}\\
&\leq \int_{y}  \sum_{ m+\beta=\alpha}  \| \big|\mathbf{1}^n_{0}(y) \big|_m|y|^{\gamma-\beta+|\mathfrak{s}|} \  \frac{\big|\Gamma_{x+y,x} f(x) -f(x+y)\big|_\beta}{|y|^{\gamma-\beta+|\mathfrak{s}|}}\|_{\L^p(dx)}\\
&\lesssim \sum_{ m+\beta=\alpha} \| \big|\mathbf{1}^n_{0}(y) \big|_m|y|^{\gamma-\beta+|\mathfrak{s}|}\|_{L^{p'}(dy)} \|f\|_{d^{\gamma,\beta}_{p,q}}\\
&\lesssim \sum_{ m+\beta=\alpha} 2^{nm} 2^{-n(\gamma-\beta+|\mathfrak{s}|)} \|f\|_{d^{\gamma,\beta}_{p,q}}\\
&= \sum_{m+\beta=\alpha} 2^{-(\gamma-\alpha+|\mathfrak{s}|)} \|f\|_{d^{\gamma,\beta}_{p,q}} \, .\\
\end{align*}
Next we establish the appropriate bound on 
\begin{align*}
f^n (x+h)-&\Gamma_{x+h,x} f^n(x) -\big(f(x+h)-\Gamma_{x+h,x} f(x) \big)\\
&=f^n (x+h) -f(x+h)+\Gamma_{x+h,x}\big( f(x)-f^n (x)  \big)\\
&=\int_{y} \mathbf{1}^n_y(x+h)\star\big(\Gamma_{x+h,y} f(y) -f(x+h)\big)-\Gamma_{x+h,x}\int_{y} \mathbf{1}^n_y(x)\star\big(\Gamma_{x,y} f(y) -f(x)\big)\\
&=\int_{y} {\mathbf{1}^n_y(x+h)}\star\big(\Gamma_{x+h,y} f(y) -f(x+h)\big)-\int_{y} \big(\Gamma_{x+h,x}\mathbf{1}^n_y(x)\big)\star \big(\Gamma_{x+h,y} f(y) -\Gamma_{x+h,x}f(x)\big)\\
&=\int_{y} \big({\mathbf{1}^n_y(x+h)-\Gamma_{x+h,x}\mathbf{1}^n_y(x)}\big)\star\big(\Gamma_{x+h,y} f(y) -f(x+h)\big) + \\
+\int_y&\big(\Gamma_{x+h,x}\mathbf{1}^n_y(x)\big)\star\big(\underbrace{\Gamma_{x+h,y} f(y) -f(x+h)-\big(\Gamma_{x+h,y} f(y) -\Gamma_{x+h,x}f(x)}_{=\Gamma_{x+h,x}f(x)-f(x+h)}\big)\big) \, .
\end{align*}
We take care of the two summands separately: we start with the second, which is easier,
\begin{align*}
& \frac{|\int_y\big(\Gamma_{x+h,x}\mathbf{1}^n_y(x)\big)\star\big(\Gamma_{x+h,x}f(x)-f(x+h)\big)|_\alpha}{|h|^{\gamma-\alpha+|\mathfrak{s}|}_\mathfrak{s}} \\ \leq &\sum_{m+\beta=\alpha}\big|\Gamma_{x+h,x}\underbrace{\int_y\mathbf{1}^n_y(x)}_{\text{given by} \int_y \mathbb{1}_0^n \text{ at the only nonvanishing degree} m=0.}\big|_m\frac{|h|^{\gamma-\beta+|\mathfrak{s}|}_\mathfrak{s}}{|h|^{\gamma-\alpha+|\mathfrak{s}|}_\mathfrak{s}}  \frac{\big|\Gamma_{x+h,x}f(x)-f(x+h)\big)|_\beta}{|h|^{\gamma-\beta+|\mathfrak{s}|}_\mathfrak{s}}\\
= &\big(\underbrace{\int \mathbf{1}_0^n }_{\lesssim 2^{-|\mathfrak{s}|n}}\big)\frac{\big|\Gamma_{x+h,x}f(x)-f(x+h)\big)|_\alpha}{|h|^{\gamma-\beta+|\mathfrak{s}|}_\mathfrak{s}} \, .
\end{align*}
Integrating this yields the desired bound.
For the first summand we proceed similarly to above:
\begin{align*}
&\Bigg\| \big|\int_{y} \big({\mathbf{1}^n_y(x+h)-\Gamma_{x+h,x}\mathbf{1}^n_y(x)}\big)\star\big(\Gamma_{x+h,y} f(y) -f(x+h)\big)\big|_\alpha\Bigg\|_{L^p(dx)}\\
&=\Bigg\|\big|\int_{y} \big({\mathbf{1}^n_{x-y}(x+h)-\Gamma_{x+h,x}\mathbf{1}^n_{x-y}(x)}\big)\star\big(\Gamma_{x+h,x-y} f(x-y) -f(x+h)\big)\big|_\alpha\Bigg\|_{L^p(dx)}\\
&\leq \sum_{m+\beta=\alpha} \int_{y}\Bigg\|\big| {\mathbf{1}^n_{x-y}(x+h)-\Gamma_{x+h,x}\mathbf{1}^n_{x-y}(x)}\big|_m \big|\Gamma_{x+h,x-y} f(x-y) -f(x+h)\big)\big|_\beta\Bigg\|_{L^p(dx)}\\
&= \sum_{m+\beta=\alpha} \int_{y}\Bigg\|\big| {\mathbf{1}^n_{x}(x+y+h)-\Gamma_{x+y+h,x+y}\mathbf{1}^n_{x}(x+y)}\big|_m \big|\Gamma_{x+y+h,x} f(x) -f(x+y+h)\big)\big|_\beta\Bigg\|_{L^p(dx)}\\
&= \sum_{m+\beta=\alpha} \int_{y}\Bigg\|\big| {\mathbf{1}^n_{x}(x+y)-\Gamma_{x+y,x+y-h}\mathbf{1}^n_{x}(x+y-h)}\big|_m \big|\Gamma_{x+y,x} f(x) -f(x+y)\big)\big|_\beta\Bigg\|_{L^p(dx)}\\
&= \sum_{m+\beta=\alpha} \int_{y}\Bigg\|\big| {\mathbf{1}^n_{0}(y)-\Gamma_{y,y-h}\mathbf{1}^n_{0}(y-h)}\big|_m \big|\Gamma_{x+y,x} f(x) -f(x+y)\big)\big|_\beta\Bigg\|_{L^p(dx)}\\
&= \sum_{m+\beta=\alpha} \int_{y}\Bigg\|\big| {\mathbf{1}^n_{0}(y)-\Gamma_{y,y-h}\mathbf{1}^n_{0}(y-h)}\big|_m \ |y|_{\mathfrak{s}}^{\gamma-\beta+|\mathfrak{s}|} \frac{\big|\Gamma_{x+y,x} f(x) -f(x+y)\big)\big|_\beta}{|y|_{\mathfrak{s}}^{\gamma-\beta+|\mathfrak{s}|}}\Bigg\|_{L^p(dx)}\\
&= \sum_{m+\beta=\alpha} \int_{y}\big| {\mathbf{1}^n_{0}(y)-\Gamma_{y,y-h}\mathbf{1}^n_{0}(y-h)}\big|_m \ |y|_{\mathfrak{s}}^{\gamma-\beta+|\mathfrak{s}|} \Bigg\|\frac{\big|\Gamma_{x+y,x} f(x) -f(x+y)\big)\big|_\beta}{|y|_{\mathfrak{s}}^{\gamma-\beta+|\mathfrak{s}|}}\Bigg\|_{L^p(dx)}\\
&\leq \sum_{m+\beta=\alpha} \Big\|\big| {\mathbf{1}^n_{0}(y)-\Gamma_{y,y-h}\mathbf{1}^n_{0}(y-h)}\big|_m \ |y|_{\mathfrak{s}}^{\gamma-\beta+|\mathfrak{s}|}\Big \|_{L^\infty(dy)} \Big\| f \Big\|_{d^{\gamma,\beta}_{p,q}}\\
&\leq \sum_{m+\beta=\alpha} 2^{-n(\gamma-\beta+|\mathfrak{s}|)}\Big\|\big| {\mathbf{1}^n_{0}(y)-\Gamma_{y,y-h}\mathbf{1}^n_{0}(y-h)}\big|_m \Big \|_{L^\infty(dy)} \Big\| f \Big\|_{d^{\gamma,\beta}_{p,q}} \, .\\
\end{align*}
If we devide the first and last line of the previous estimates by $|h|_\mathfrak{s}^{\gamma-\alpha+|\mathfrak{s}|}$ and integrating
\begin{align*}
&\int_h \Big(\frac{\big\| \big|\int_{y} \big({\mathbf{1}^n_y(x+h)-\Gamma_{x+h,x}\mathbf{1}^n_y(x)}\big)\star\big(\Gamma_{x+h,y} f(y) -f(x+h)\big)\big|_\alpha\big\|_{L^p(dx)}}{\|h\|_\mathfrak{s}^{\gamma-\alpha}}\Big)^q \frac{dh}{|h|_{\mathfrak{s}}^{|\mathfrak{s}|}}\\
&\lesssim \sum_{m+\beta=\alpha} 2^{-qn(\gamma-\beta+|\mathfrak{s}|)} \int_h \Big(\frac{\big\|\big| {\mathbf{1}^n_{0}(y)-\Gamma_{y,y-h}\mathbf{1}^n_{0}(y-h)}\big|_m \big \|_{L^\infty(dy)}}{{\|h\|_\mathfrak{s}^{\gamma-\alpha}}}\Big)^q \frac{dh}{|h|_{\mathfrak{s}}^{|\mathfrak{s}|}} \Big\| f \Big\|^q_{d^{\gamma,\beta}_{p,q}}\\
&= \sum_{m+\beta=\alpha} 2^{-qn(\gamma-\beta+|\mathfrak{s}|)} \int_h \Big(\frac{\big\|\big| {\mathbf{1}^n_{0}(y)-\Gamma_{y,y-h}\mathbf{1}^n_{0}(y-h)}\big|_m \big \|_{L^\infty(dy)}}{{\|h\|_\mathfrak{s}^{(\gamma-\beta+\epsilon)-m}}}\Big)^q \frac{dh}{|h|_{\mathfrak{s}}^{|\mathfrak{s}|-q\epsilon}} \Big\| f \Big\|^q_{d^{\gamma,\beta}_{p,q}}\\
&\leq \sum_{m+\beta=\alpha} 2^{-qn(\gamma-\beta+|\mathfrak{s}|)} 
\Big\| \mathbf{1}^n_{0} \Big\|^q_{d^{\gamma-\beta+\epsilon,m}_{\infty,\infty}}
 \underbrace{\int_{|h|_\mathfrak{s}\leq C 2^{-n}}\frac{dh}{|h|_{\mathfrak{s}}^{|\mathfrak{s}|-q\epsilon}}}_{\lesssim 2^{-nq\epsilon}} \Big\| f \Big\|^q_{d^{\gamma,\beta}_{p,q}}\\
 &\leq \sum_{m+\beta=\alpha} 2^{-qn(\gamma-\beta+|\mathfrak{s}|+\epsilon)} 
\Big\| \mathbf{1}^n_{0} \Big\|^q_{d^{\gamma-\beta+\epsilon,m}_{\infty,\infty}}
\Big\| f \Big\|^q_{d^{\gamma,\beta}_{p,q}} \, .\\
\end{align*}
For $\phi\in \mathcal{C}_c^\infty$ and $\lambda\in (0,1)$  the scaling  property $\|\Phi(\frac{\cdot}{\lambda})\|_{d^{\gamma,m}_{\infty,\infty}}\lesssim \lambda^{-\gamma}\|\Phi\|_{d^{\gamma,m}_{\infty,\infty}}$ holds with the implicit constant only depending on the support of $\phi$. (Again $\Phi$ denotes the canonical lift of $\phi$.) Using this we obtain
\begin{align*}
&\int_h \Big(\frac{\big\| \big|\int_{y} \big({\mathbf{1}^n_y(x+h)-\Gamma_{x+h,x}\mathbf{1}^n_y(x)}\big)\big(\Gamma_{x+h,y} f(y) -f(x+h)\big)\big|_\alpha\big\|_{L^p(dx)}}{\|h\|_\mathfrak{s}^{\gamma-\alpha}}\Big)^q \frac{dh}{|h|_{\mathfrak{s}}^{|\mathfrak{s}|}}\\
&\lesssim \sum_{m+\beta=\alpha} 2^{-qn(\gamma-\beta+|\mathfrak{s}|+\epsilon)} 2^{qn(\gamma-\beta+\epsilon)}
\Big\| \mathbf{1}_{0} \Big\|^q_{d^{\gamma-\beta+\epsilon,m}_{\infty,\infty}}
\Big\| f \Big\|^q_{d^{\gamma,\beta}_{p,q}}\\
&=\sum_{m+\beta=\alpha} 2^{-qn|\mathfrak{s}|} 
\Big\| \mathbf{1}_{0} \Big\|^q_{d^{\gamma-\beta+\epsilon,m}_{\infty,\infty}}
\Big\| f \Big\|^q_{d^{\gamma,\beta}_{p,q}} \, .\\
\end{align*}
\end{proof}
\begin{remark}
Let us note the interesting fact that the convergence rate of this approximation depends exponentially on the effective dimension $|\mathfrak{s}|$.
\end{remark}
%\bibliography{quellen}{}
%\bibliographystyle{amsalpha}
\providecommand{\bysame}{\leavevmode\hbox to3em{\hrulefill}\thinspace}
\providecommand{\MR}{\relax\ifhmode\unskip\space\fi MR }
\providecommand{\MRhref}[2]{%
  \href{http://www.ams.org/mathscinet-getitem?mr=#1}{#2}
}
\providecommand{\href}[2]{#2}

\end{document}